\documentclass[11pt]{article}
\usepackage{amsmath, amsthm, amssymb, latexsym}
\usepackage{graphicx}
\usepackage{amsmath}
\usepackage{amsfonts}
\usepackage{enumerate}
\usepackage{latexsym}
\usepackage{epsfig}
\usepackage{amsthm}  
\usepackage{amsmath}
\usepackage{amssymb} 
\usepackage{latexsym}
\usepackage{epsfig} 
\usepackage{amssymb,latexsym}
\usepackage[all]{xy}

\usepackage{a4,fullpage,amssymb,epsf,psfrag,times}

\newcommand     {\comment}[1]   {}
\newcommand{\mute}[2] {}
\newcommand     {\printname}[1] {}

\newtheorem{theorem}{Theorem}[section]
\newtheorem{proposition}[theorem]{Proposition}

\newtheorem{lemma}[theorem]{Lemma}
\newtheorem{claim}[theorem]{Claim}

\newtheorem{corollary}[theorem]{Corollary}

\newcommand{\id}{\mathop{{\rm Id}}\nolimits}

\newcommand{\PSL}{\mathop{{\rm PSL}}\nolimits}

\newcommand{\Aut}{\mathop{{\rm Aut}}\nolimits}
\newcommand{\Ham}{\mathop{{\rm Ham}}\nolimits}
\newcommand{\Emb}{\mathop{{\rm Emb}}\nolimits}
\newcommand{\Diff}{\mathop{{\rm Sympl}}\nolimits}
\newcommand{\Sympl}{\mathop{{\rm Sympl}}\nolimits}
\newcommand{\Dif}{\mathop{{\rm Diff}}\nolimits}

\newcommand{\closed}{\mathop{{\rm closed}}\nolimits}
\newcommand{\exact}{\mathop{{\rm exact}}\nolimits}
\newcommand{\red}{\mathop{{\rm red}}\nolimits}
\newcommand{\R}{{\mathbb R}}

\newcommand{\J}{{\mathcal J}}
\newcommand{\M}{{\mathcal M}}

\newcommand{\X}{{\mathcal X}}
\newcommand{\C}{{\mathbb C}}

\newcommand{\Z}{{\mathbb Z}}

\def \g {{\mathfrak h \mathfrak a \mathfrak m}}

\newcommand{\op}[1]{\!\!\mathop{\rm ~#1}\nolimits}

\newcommand{\scriptop}[1]{\!\!\mathop{\mbox{\rm \scriptsize ~#1}}\nolimits}

\newenvironment{remark}{\refstepcounter{theorem}\par\medskip\noindent{\bf
Remark~\thetheorem~~}}{\unskip\nobreak\hfill\hbox{ $\oslash$}\par\bigskip}

\newenvironment{noTitle}{\refstepcounter{theorem}\par\medskip\noindent{\thetheorem~~}}{\unskip\nobreak\hfill\hbox{ $\oslash$}\par\bigskip}

\newenvironment{definition}{\refstepcounter{theorem}\par\medskip\noindent{\bf
Definition~\thetheorem~~}}{\unskip\nobreak\hfill\hbox{ $\oslash$}\par\bigskip}

\begin{document}

\title{Symplectic forms on the space of embedded symplectic surfaces and their reductions}
\author{Liat Kessler} 

\date{}

\maketitle

\begin{abstract}
 Let $(M,\, \omega)$ be a symplectic manifold, and $(\Sigma, \sigma)$ a closed connected symplectic $2$-manifold. We construct a weakly symplectic form ${\omega^{D}}_{(\Sigma, \, \sigma)}$ on the 
 space of immersions $\Sigma \to M$ that is a special case of Donaldson's form. We show that the restriction of ${\omega^{D}}_{(\Sigma, \, \sigma)}$ to any orbit of the group of Hamiltonian symplectomorphisms through a symplectic embedding $(\Sigma, \, \sigma) \hookrightarrow (M,\, \omega)$ descends to a weakly symplectic form $\omega^D_{\red}$ on the quotient by $\Diff(\Sigma, \, \sigma)$, and that the obtained symplectic space  is a symplectic quotient of  the subspace of symplectic embeddings $\mathcal{S}_{\op{e}}(\Sigma, \, \sigma)$
with respect to the $\Diff(\Sigma, \, \sigma)$-action. We also compare ${\omega^{D}}_{(\Sigma, \, \sigma)}$ and its reduction $\omega^D_{\red}$ to another $2$-form on the space of immersed symplectic $\Sigma$-surfaces in $M$. We conclude by a result on the restriction of  ${\omega^{D}}_{(\Sigma, \, \sigma)}$ to moduli spaces of $J$-holomorphic curves.
  \end{abstract}

\section{Introduction}
Let $(M, \, \, \omega)$ be a compact finite-dimensional symplectic manifold, and $\Sigma$ a closed connected $2$-manifold. Fix  a symplectic form $\sigma$ on $\Sigma$. 
We identify the tangent space to 
$\op{C}^{\infty}(\Sigma,\, M)$ at $f \colon \Sigma \to M$
with the space $\Omega^0(\Sigma,\, f^*(\op{T}\!M))$ 
of  smooth vector fields $\tau \colon \Sigma \rightarrow f^{\ast} (\op{T}\!M)$. 

\begin{definition}
Define a $2$\--form  on $\op{C}^{\infty}(\Sigma,\, M)$ by
  \begin{eqnarray} \label{omegas2:def}
       ({\omega^{D}}_{(\Sigma, \, \sigma)})_f( \tau_1, \,\tau_2):=\int_{\Sigma} \omega_{f(x)}(\tau_1(x),\tau_2(x)) \, \sigma, \nonumber
  \end{eqnarray}
  where  $\tau_1, \tau_2 \in \op{T}_f(\op{C}^{\infty}(\Sigma,\, M))$. 
\end{definition}
The form ${\omega^D}_{(\Sigma, \, \sigma)}$ is a special case of  the  two-form  on the space of smooth maps $S  \to M$ of a compact oriented manifold $S$ equipped with a fixed volume form $\eta$, introduced by Donaldson in \cite{donaldson}. 
Under some topological conditions, e.g., that $H^{1}(S)=0$ and  for all $i \in \op{C}^{\infty}(\Sigma,\, M)$ the class $i^{*}[\omega]$ is the zero class in $H^{2}(S)$,
Donaldson described a moment map for the action of the Lie group of volume preserving diffeomorphisms $\Dif(S, \, \eta)$ on $\op{C}^{\infty}(S,\,M)$. This action restricts to a Hamiltonian action on the subspace of embeddings $\Emb(S,\,M)$.
In \cite{lee}, Brian Lee gives a rigorous formulation of Donaldson's heuristic construction, in the ``Convenient Setup'' of Fr\"olicher, Kriegl, and Michor \cite{conv}, and shows that the form reduces to the image of $\Ham(M, \, \, \omega)$-orbits 
through isotropic embeddings in $\Emb(S,\,M)$ under the projection to the quotient $\Emb(S,\, M) / \Dif(S, \, \eta)$. 
Lee's result does not assume $H^{1}(S)=0$. In this paper, we will omit also the condition $i^{*}[\omega]=0$, and instead of looking at orbits through isotropic embeddings, we will look at orbits through symplectic embeddings. 
Denote by
 $$\mathcal{S}_{\op{e}}(\Sigma, \, \sigma)$$ the subspace of symplectic embeddings $(\Sigma, \, \sigma) \to (M, \, \, \omega)$. 
 The Lie group 
 $\Diff(\Sigma, \, \sigma)$ of diffeomorphisms  of $\Sigma$ that pull back $\sigma$ to itself 
acts freely on  $\mathcal{S}_{\op{e}}(\Sigma, \, \sigma)$ on the right. 
In this paper we study the reduction of ${\omega^{D}}_{(\Sigma, \, \sigma)}$ to $\mathcal{S}_{\op{e}}(\Sigma, \, \sigma)$ modulo $\Sympl(\Sigma, \, \sigma)$ and to moduli spaces of un-parametrized $J$-holomorphic curves. The terms smooth manifold and map, tangent space, and differential form are interpreted in the ``Convenient Setup''. In this framework, the local model is the convenient vector space: a locally convex vector space $E$ with the property that for any smooth (infinitely differentiable) curve  $c_1 \colon \R \to E$ there is a curve $c_2 \colon \R \to E$ such that ${c_2}'={c_1}$, with the $c^{\infty}$-topology: the finest topology for which all smooth curves $\R \to E$ are continuous. (The $c^{\infty}$-topology is finer than the locally convex topology 
on $E$. If $E$ is a Frechet space, (i.e., a complete and metrizable locally convex space), then 
the two topologies coincide.) 
A map between convenient vector spaces is smooth if it sends smooth curves to 
smooth curves. Smooth manifolds are modeled on convenient 
vector spaces via charts, whose transition functions are smooth; a map between smooth manifolds is smooth if it maps smooth curves to smooth curves. (See 
 \cite{conv} and \cite[Sec. 2]{lee}.)

In 
the Appendix we show that the $2$\--form ${\omega^{D}}_{(\Sigma, \, \sigma)}$ is closed and its restriction to the space of immersions $\Sigma \to M$ is weakly non-degenerate. 
We also show that for an almost complex structure $J \colon \op{T}\!M \to \op{T}\!M$ that is compatible with $\omega$, the induced almost complex structure $
  \tilde{J}: \op{T}\!  \op{C}^{\infty}(\Sigma,\, M) \rightarrow 
  \op{T}\! \op{C}^{\infty}(\Sigma,\, M)
  $ is compatible with ${\omega^{D}}_{(\Sigma, \, \sigma)}$. 
 In Section \ref{sec2} we prove  that $\mathcal{S}_{\op{e}}(\Sigma, \, \sigma)$  is a smooth manifold and describe its tangent bundle, see Proposition \ref{model}; we show that the restriction of the form  ${\omega^D}_{(\Sigma, \, \sigma)}$ to $\mathcal{S}_{\op{e}}(\Sigma, \, \sigma)$ is weakly symplectic, see Proposition \ref{sympl3}. In Section \ref{sec3} we prove the following theorem. 
 \begin{theorem} \label{thm1}
 Let $\mathcal{N}$ be a 
$\Ham(M, \, \, \omega)$-orbit in  $\mathcal{S}_{\op{e}}(\Sigma, \, \sigma)$. The restriction of ${\omega^D}_{(\Sigma, \, \sigma)}$ to $\mathcal{N}$ descends to a closed weakly non-degenerate $2$-form $\omega^D_{\red}$ on the image $\mathcal{O}$ in the orbit space under the projection $ q \colon \mathcal{S}_{\op{e}}(\Sigma, \, \sigma) \to \mathcal{S}_{\op{e}}(\Sigma, \, \sigma) / \Diff(\Sigma, \, \sigma).$ The symplectic space $(\mathcal{O},\, \, \omega^D_{\red})$ is a symplectic quotient of $\mathcal{S}_{\op{e}}(\Sigma, \, \sigma)$
with respect to the $\Diff(\Sigma, \, \sigma)$-action. 
\end{theorem}
The notion of a symplectic quotient here 
does not depend on having a moment map, see Definition \ref{leedef}.
It is motivated by the optimal reduction method of Ortega and Ratiu \cite{ratiu}. 

We also compare ${\omega^{D}}_{(\Sigma, \, \sigma)}$ to 
the $2$-form we defined in \cite{ckp} on the space of immersed symplectic $\Sigma$-surfaces in $M$.
Denote by
$$\op{ev} \colon \op{C}^{\infty}(\Sigma,\, M)  \times \Sigma \to M$$ 
the \emph{evaluation map} 
$$
\op{ev}(f,\,x):= f(x).
$$

\begin{definition}
Define a $2$-form on $\op{C}^{\infty}(\Sigma,\, M)$ as the push-forward of the $4$-form ${\op{ev}}^{\ast}(\omega \wedge \omega)$ along the coordinate-projection  
$\pi_{\op{C}^{\infty}(\Sigma,\, M)} \colon \op{C}^{\infty}(\Sigma,\, M) \times \Sigma \to \op{C}^{\infty}(\Sigma,\, M)$ by
  \begin{eqnarray} \label{omegas1:def}
       (\omega_{\op{C}^{\infty}(\Sigma,\, M)})_f( \tau_1, \,\tau_2):=\int_{\{f\} \times \Sigma} \iota_{( \ell_1 \wedge
    \ell_2 )} {\op{ev}}^{\ast} ( \omega \wedge \omega ).
  \end{eqnarray}
Here $\ell_{i} \in \op{T}(\op{C}^{\infty}(\Sigma,\, M) \times \Sigma)$ is a
\emph{lifting} of $\tau_{i} \in \op{T}_f(\op{C}^{\infty}(\Sigma,\, M))$,  i.e.,
$$
\op{d} ( \pi_{\op{C}^{\infty}(\Sigma,\, M)} ) {\ell_{i}}_{( f,\, x )}  = \tau_{i} \,\,\, \textup{at each point}\,\,\,
(f, \,x) \in \pi_{\op{C}^{\infty}(\Sigma,\, M)}^{- 1} (f).
$$
\end{definition}
Denote 
\begin{eqnarray} 
\mathcal{S}_{\op{i}}(\Sigma):=\{ f \colon \Sigma \to M \ \mid \ f \textup{ is an
immersion, }f^{*}\omega \textup{ is a symplectic form on }\Sigma \}. \nonumber
\end{eqnarray}
The space $\mathcal{S}_{\op{i}}(\Sigma)$ is  an open subset of 
$\op{C}^{\infty}(\Sigma,\, M)$ in the $C^{\infty}$-topology. 
Let $$\omega_{\mathcal{S}_{\op{i}}(\Sigma)}$$ be the $2$\--form on  $\mathcal{S}_{\op{i}}(\Sigma)$ given
by the restriction of $\omega_{\op{C}^{\infty}(\Sigma,\, M)}$.
We showed in \cite{ckp} that 
the $2$\--form  $\omega_{\op{C}^{\infty}(\Sigma,\, M)}$ on $\op{C}^{\infty}(\Sigma,\, M)$ is well defined and closed, and $\omega_{\op{C}^{\infty}(\Sigma,\, M)} (\tau, \,\cdot )$
vanishes at $f$ if $\tau$ is  everywhere tangent to $f(\Sigma)$. Furthermore,
$$
 \omega_{\mathcal{S}_{\op{i}}(\Sigma)} (\tau, \,\cdot ) = 0 \,\,\textup{at}\,\, f \,\, \iff
\tau \textup{ is tangent to }f(\Sigma)  \,\,\textup{at every}\,\, x \in \Sigma. 
  $$
We 
say that a vector
field  $\tau \colon \Sigma \rightarrow f^{\ast} (\op{T}\!M)$ {\em{is tangent to $f(\Sigma)$ at $x$}} if $\tau(x) \in
\op{d}\!f_{x}(\op{T}_{x}\Sigma)$. See also \cite{secondlee}.
 
 Consider the space of $\omega$-compatible almost complex structures $\J=\J(M, \, \omega)$ on $(M, \, \omega)$.  Fix $\Sigma=(\Sigma,\, j)$, where $j$ is a complex structure on $\Sigma$. 
The moduli space
$\M_{\scriptop{i}}(A,\,\Sigma,\,J)$  is the space of simple immersed $(j,J)$-holomorphic $\Sigma$-curves in a homology class $A \in \op{H}_2(M,\,\Z)$. 
The moduli space
$\M_{\scriptop{e}}(A,\,\Sigma,\,J)$  is the space of embedded $(j,J)$-holomorphic $\Sigma$-curves in a homology class $A \in \op{H}_2(M,\,\Z)$. 
We look at almost complex structures that are regular for the projection map $$p_{A} \colon \M_{\scriptop{i}}(A,\,\Sigma, \,\mathcal{J}) \to \J;$$ for such a  $J$, the spaces $\M_{\scriptop{i}}(A,\,\Sigma,\,J)$  and $\M_{\scriptop{e}}(A,\,\Sigma,\,J)$ are finite-dimensional manifolds. 
(The set of $p_{A}$-regular $\omega$\--compatible almost complex structures is of the second category in $\J$.) 
See \cite[Thm 3.1.5]{MS2}. There is merit to the form  $\omega_{\mathcal{S}_{\op{i}}(\Sigma)}$ in the fact that it is degenerate along directions tangent to $f(\Sigma)$, hence
descends to a well defined form on the quotient space  ${\widetilde{\M}}_{\scriptop{i}}(A,\,\Sigma,\,J)$ of  $\M_{\scriptop{i}}(A,\,\Sigma,\,J)$ by the proper action of the group  $\Aut(\Sigma, \, j)$ of bi-holomorphisms of $\Sigma$: this enables us to apply Gromov's compactness theorem 
and get a well defined invariant of $(M, \, \omega)$.
If $J_* \in \mathcal{J}_{\scriptop{reg}}(A)$ is integrable, then the restriction of the form $\omega_{\mathcal{S}_{\op{i}}(\Sigma)}$ to $\M_{\scriptop{i}}(A,\,\Sigma,\,J_*)$ is non-degenerate, up to reparametrizations; see \cite[Prop. 4.4]{ckp}. We obtained results on the existence of $J$-holomorphic curves in a homology class $A$ for some subset of $\J$, and in some cases for a generic $J$, see \cite[Cor. 1.3]{ckp}. 

Here we show that the $2$-forms $2{\omega^{D}}_{(\Sigma, \, \sigma)}$ and  $\omega_{\mathcal{S}_{\op{i}}(\Sigma)}$ coincide in exact direction, hence on the quotient of a $\Ham(M, \, \, \omega)$-orbit with respect to the $\Diff(\Sigma, \, \sigma)$-action. 
The difference between ${\omega^D}$ and $\omega_{\mathcal{S}_{\op{i}}(\Sigma)}$  is that $\iota_{v}{\omega^D}$ is degenerate along vectors in $\op{T}\!\mathcal{S}_{\op{e}}(\Sigma, \, \sigma)$ everywhere tangent to $\Sigma$ iff $v=i^{*}{V_{H}}$ for a Hamiltonian vector field $V_{H}$ on $M$ whereas $\iota_{v} { {\omega_{\mathcal{S}_{\op{i}}(\Sigma)}}}$ is degenerate along vectors everywhere tangent to $\Sigma$ for every $v$ in  $\op{T}\!\mathcal{S}_{\op{i}}(\Sigma)$, see Remark \ref{remold}. Due to this difference, Theorem \ref{thm1} does not hold for  $\omega_{\mathcal{S}_{\op{i}}(\Sigma)}$. On the other hand, we do not get a well defined reduction of $\omega^D_{(\Sigma, \, \sigma)}$ on the quotient of  $\M_{\scriptop{i}}(A,\,\Sigma,\,J)$ by the action of $\Aut(\Sigma, \, j)$, as we did for $\omega_{\mathcal{S}_{\op{i}}(\Sigma)}$. However we do get a partial result.  
For $J \in \J(M,\, \omega)$, denote by
  $$\Ham^{J}(M, \, \omega)$$ the subgroup of $\Ham(M, \, \omega)$ of $J$-holomorphic Hamiltonian symplectomorphisms.
    Let $\mathcal{N}$ be an orbit of $\Ham^{J}(M, \, \omega)$ through an embedded $(j,J)$-holomorphic curve $f \colon \Sigma \to M$ for which $f^{*}\omega=\sigma$. The orbit $\mathcal{N}$ is a subset of   
  $\M_{\scriptop{e}}(A,\,\Sigma,\,J)$, where $$A \in \op{H}_2(M,\,\Z)$$ is the class for which the area $f^{*}\omega(\Sigma)=\sigma(\Sigma)$ for (every) $f \in A$.

\begin{corollary} \label{mod}
Assume that the symplectic form $\sigma$  on $\Sigma$ is compatible with the complex structure $j$  on $\Sigma$.
Let $J \in \J(M,\, \omega)$, assume that $J$ is integrable and regular for $A$.
Let $\mathcal{N}$ be an orbit of $\Ham^{J}(M, \, \omega)$ through a $(j,J)$-holomorphic embedding $f \colon \Sigma \to M$ for which $f^{*}\omega=\sigma$. 
   
The forms  ${\omega^{D}_{(\Sigma, \, \sigma)}}$ and $\omega_{{\mathcal{S}_{\op{i}}(\Sigma)}}$ descend to well defined symplectic forms  $\omega^D_{\red}$ and  $\omega^{\red}_{\mathcal{S}_{\op{i}}(\Sigma)}$ on the quotient of $\mathcal{N}$  with respect to $\Aut(\Sigma,\,j)$. The form $\omega^{\red}_{\mathcal{S}_{\op{i}}(\Sigma)}$ coincides with the form $2\omega^D_{\red}$ on the quotient.
\end{corollary}

 \section{The space $\mathcal{S}_{\op{e}}(\Sigma, \, \sigma)$} \label{sec2}

 \begin{proposition} \label{thclonon2}
The $2$\--form ${\omega^{D}}_{(\Sigma, \, \sigma)}$ on $\op{C}^{\infty}(\Sigma,\, M)$ is closed and its restriction to the
space of immersions $\Sigma \to M$ is weakly non-degenerate. 
\end{proposition}

The space  $\op{C}^{\infty}(\Sigma,\, M)$ is a smooth manifold in the Convenient Setup, 
modeled on spaces $\Gamma(f^{*}\op{T}\!M)$ of sections of the pullback bundle along 
$f \in \op{C}^{\infty}(\Sigma,\, M)$ \cite[42.1]{conv}. The space $\Gamma(f^{*}\op{T}\!M)$ has a natural convenient structure \cite[30.1]{conv}.

 A $2$\--form $\Omega$ on a manifold $X$ (possibly infinite-dimensional) is called \emph{weakly non-degenerate} if for every $x \in X$ and $0 \neq v \in \op{T}_{x}X$ there exists 
a $w \in \op{T}_{x}X$ such that $\Omega_{x}(v, w) \neq 0$. This is equivalent to its associated vector bundle homomorphism $\Omega^{\flat}\colon \op{T}\!X \to \op{T}^{*}\!X$ being injective. 
 If $\Omega^{\flat} \colon \op{T}\!X \to \op{T}^{*}\!X$ is an isomorphism, i.e., invertible with a smooth inverse, 
then $\Omega$ is called \emph{strongly non-degenerate}. In this paper, by non-degenerate we mean weakly non-degenerate.
If $\Omega$ is closed and weakly non-degenerate, it is called \emph{weakly symplectic}.

For the proof of Proposition \ref{thclonon2} and required facts on compatible almost complex structures, see the Appendix.
 
{\bf Notation: }

For every embedding $i \colon \Sigma \to M$, 
for $v \in \Gamma(i^{*}\op{T}\!M)$, let $\alpha_v \in \Omega^1(\Sigma)$ denote the form
$$
(\alpha_v)_{x}(\xi) := \omega_{x}(v(x), \op{d}\!i_{x}\xi) \, \text{for } \xi \in \op{T}_{x}\Sigma. 
$$
Also, set 
$$
\Gamma_{\closed}(i^{*}\op{T}\!M):= \{v \in \Gamma(i^{*}\op{T}\!M) \, | \, \alpha_v \text{ is a closed }1\text{-form on }\Sigma\},
$$ 
and
 $$
\Gamma_{\exact}(i^{*}\op{T}\!M):= \{v \in \Gamma(i^{*}\op{T}\!M) \, | \, \alpha_v \text{ is an exact }1\text{-form on }\Sigma\}.
$$

For a vector field 
$$v \in \op{T}_{i}C^{\infty}(\Sigma, \, M)$$
denote by 
\begin{equation}\label{decomp}
\xi_v+\tau_v
\end{equation}
the decomposition of $v$ to a vector field $\xi_v$ everywhere $\omega$-orthogonal to $\Sigma$ and a vector field $\tau_v$ everywhere tangent to $\Sigma$.
Such a decomposition exists and is unique, e.g., by Remark \ref{extend} and Corollary \ref{corj}.

We 
say that a vector
field  $\tau \colon \Sigma \rightarrow \op{T}\!M$ is {\em{tangent to $\Sigma$ at $x$}} if $\tau(x) \in
\op{T}_{i(x)}i(\Sigma)$. 
We say that a vector field $\xi \colon \Sigma \rightarrow \op{T}\!M$ is {\em{$\omega$-orthogonal to $\Sigma$ at $x$}} if $\xi(x) \in  (\op{T}_{i(x)}i(\Sigma))^{\omega}$.

  Denote by
 $$\mathcal{S}_{\op{e}}(\Sigma, \, \sigma)$$ the set of embeddings $(\Sigma, \, \sigma) \to (M, \, \omega)$ such that  $i^{*}\omega=\sigma$.

\begin{proposition} \label{model}
The set $\mathcal{S}_{\op{e}}(\Sigma, \, \sigma)$ is a 
smooth manifold
modeled on  $\Gamma_{\exact}(i^{*}\op{T}\!M) \oplus \X(\Sigma, \, \sigma)$, where  $$\X(\Sigma, \, \sigma)=\{\xi \text{ a vector field on }\Sigma\, | \, \mathcal{L}_{\xi}\sigma=0\}.$$ 
\end{proposition}  

To prove the proposition, we first recall the symplectic tubular neighbourhood Theorem of Weinstein. 

\begin{noTitle} \label{normal}
Consider a symplectic embedding $i \colon (\Sigma, \, \sigma) \hookrightarrow (M,\, \omega)$.
The symplectic normal bundle $$\op{N}\! \Sigma=\{(x,v) \, | \, x \in \Sigma, \, v \in \op{T}_{i(x)}M / \op{T}_{x}i( \Sigma)\} \to \Sigma.$$ 
The \emph{minimal coupling form}, due to Sternberg \cite{stern}, is a closed $2$-form $\omega_{\op{N}\!\Sigma}$ with the following properties:
\begin{enumerate}
\item Its pullback to the fibers coincide with the fiberwise symplectic forms.
\item Its pullback to the zero section coincides with $\sigma$.
\item At the points of the zero section, the fibers of $\op{N}\! \Sigma$ are $\omega_{\op{N}\!\Sigma}$-orthogonal to the zero section.
\end{enumerate}
Consequently, $\omega_{\op{N}\!\Sigma}$ is non-degenerate near the zero section.

The symplectic normal bundle $\op{N}\!\Sigma$ can be realized as a subbundle of $\op{T\!}M$:
the symplectic orthocomplement of $\op{T}\!\Sigma=\op{T}\!i(\Sigma)$ in $\op{T}\!M|_{i(\Sigma)}$. In other words, the fiber $\op{N}\!_x \Sigma$ at $x \in \Sigma$  is identified with 
$$(\op{T}_{i(x)}i(\Sigma))^{\omega}=\{v \in \op{T}_{i(x)}M \, | \, \omega(v ,  w) =0 \text{ for every }w \in \op{T}_{i(x)}i(\Sigma) \}$$ with the symplectic form $\omega |_{(\op{T}_{i(x)}i(\Sigma))^{\omega}}$.

By the classical tubular neighbourhood theorem in differential topology combined with a theorem of Weinstein \cite[Theorem 4.1]{Weinstein}, there exists a neighbourhood $U$ of the zero section in $\op{N}\!\Sigma$ and a symplectic open embedding 
\begin{equation} \label{eqnormal}
\Phi_i \colon (U, \, \omega_{\op{N}\!\Sigma}) \to (M, \, \omega)
\end{equation}
 whose restriction to the zero section is $i$, and whose differential is $\op{d}\!i$ at every point of $\Sigma$.
\end{noTitle}

\begin{lemma}\label{path}
Let $\Sigma=(\Sigma, \, \sigma) \overset{i}{\hookrightarrow}  (M, \, \omega)$ be an embedded closed connected symplectic submanifold of dimension $2$. Let 
$v \in \op{T}_{i}C^{\infty}(\Sigma, \, M)$. If $v$ is everywhere $\omega$-orthogonal to $\Sigma$, then  $v$ equals the restriction $i^{*}{V_H}$ to $i(\Sigma)$ of a Hamiltonian vector field $V_{H}$ on $M$.
\end{lemma}

\begin{proof}
By \S \ref{normal},
we can consider $\xi_v$ as a vector field $\xi_0$ on the zero section in $\op{N}\!\Sigma$;  it is enough to show that $\xi_0$ extends to a Hamiltonian vector field $\xi$ on a neighbourhood of the zero section in $\op{N}\!\Sigma$, since then the push forward of $\xi$  via $\Phi_i$ in \eqref{eqnormal} is a Hamiltonian vector field  on a neighbourhood of $\Sigma$ in $M$.  Then $\xi_v$ can be extended to a Hamiltonian vector field 
on $M$, using a cut-off function with a support that is close enough to $\Sigma$. 

By assumption, for every $x$ in the zero section, $\xi_0(x)$ is in $\op{N}\!_{x}\Sigma=(\op{T}_{i(x)}i(\Sigma))^{\omega}$. 
Each of the fibers 
$$((\op{T}_{i(x)}i(\Sigma))^{\omega} ,\, \omega_{\op{N}\! \Sigma} |_{(\op{T}_{i(x)}i(\Sigma))^{\omega}})=((\op{T}_{i(x)}i(\Sigma))^{\omega} ,\, \omega |_{(\op{T}_{i(x)}i(\Sigma))^{\omega}})$$ 
is a symplectic vector space; to each vector $\xi_0(x) \in (\op{T}_{i(x)}i(\Sigma))^{\omega}$ there corresponds a linear function $\omega_{\op{N}\! \Sigma} (\xi_0(x),\cdot)=\omega(\xi_0(x),\cdot)$ from the fiber to $\R$.  Taking the union of these functions over the points of the zero section, we get a function 
\begin{equation}  \label{h}
H \colon \op{N}\! \Sigma \to \R
\end{equation}
which is smooth in a neighbourhood of the zero section in  $(\op{N}\! \Sigma , \, \omega_{\op{N}\! \Sigma})$. Take $\xi$ to be the vector field defined by $$\op{d} \! H=\omega_{\op{N}\! \Sigma}(\xi,\cdot)$$ in a neighbourhood of the zero section on which $\omega_{\op{N}\! \Sigma}$ is non-degenerate.  
\end{proof}

\begin{lemma} \label{vanish}
Let $\Sigma=(\Sigma, \, \sigma) \overset{i}{\hookrightarrow}  (M, \, \omega)$ be an embedded closed connected symplectic submanifold of dimension $2$. Let 
$v \in \Gamma_{\closed}(i^{*}\op{T}\!M)$.
The following are equivalent.
\begin{enumerate}
\item The vector field $v$ equals the restriction $i^{*}{V_H}$ to $i(\Sigma)$ of a Hamiltonian vector field $V_{H}$ on $M$.
\item The form $\alpha_v=\omega(v, \op{d}\!i(\cdot))$ on $\Sigma$
is exact. 
\item $ {{({\omega^D}_{(\Sigma, \, \sigma)})}}_{i}(v,w)  =0$ for every $w$ that is everywhere tangent to $\Sigma$ and satisfies $\mathcal{L}_{{(\op{d}\!i)}^{-1}w}\sigma=0$.
\item $v$ is everywhere $\omega$-orthogonal to $\Sigma$.
\end{enumerate} 
\end{lemma}
Recall that a vector field $X$ on $M$ is Hamiltonian if the form $\iota_{X}\omega$ is exact. Here $\op{d}\!i \colon \op{T}\!\Sigma \to \op{T}\!(i (\Sigma))$.

\begin{proof}

\begin{itemize}
\item  [ ]
\item [ $1 \Rightarrow 2$] If $V_H$ is a Hamiltonian vector field on $M$, then on $\Sigma$ the form $\alpha_{i^{*}{V_H}}$ equals $\op{d} \! h$ with $h=H \circ i$.
\item [ $2 \Rightarrow 3$]  If $\alpha_v = \op{d} \! h$ for a function $h \colon \Sigma \to \R$ then for every $w$ everywhere tangent to $\Sigma$ such that $(\op{d}\!i)^{-1}w \in \X(\Sigma, \, \sigma)$, 
 \begin{eqnarray}  
{{({\omega^D}_{(\Sigma, \, \sigma)})}}_{i}(v,w) &=& \int_{\Sigma} {\omega(v,w) \sigma}=\int_{\Sigma} {\op{d} \! h((\op{d}\!i)^{-1}w ) \sigma} \nonumber \\
                                           &=& \int_{\Sigma}(\mathcal{L}_{(\op{d}\!i)^{-1}w } h) \sigma=\int_{\Sigma}\mathcal{L}_{(\op{d}\!i)^{-1}w } (h \sigma)=\mathcal{L}_{(\op{d}\!i)^{-1}w } \int_{\Sigma} h\sigma \nonumber \\
                                           &=& \lim_{t \to 0} \frac{{\phi_t}^{*}\int_{\Sigma}h\sigma-\int _{\Sigma} h \sigma}{t}=0. \label{new}
\end{eqnarray}
The fourth equality is since $\mathcal{L}_{(\op{d}\!i)^{-1}w }\sigma=0$ and the fact that $\mathcal{L}_{(\op{d}\!i)^{-1}w }(h \sigma)=(\mathcal{L}_{(\op{d}\!i)^{-1}w }h)\sigma+h(\mathcal{L}_{(\op{d}\!i)^{-1}w }\sigma)$; the fifth equality is since $\Sigma$ is compact, and the last equality is since for an orientation preserving integral curve $t \to \phi_t$ and a $2$-form $\gamma$ on $\Sigma$, $\int_{\Sigma} \gamma$ is invariant under pulling back by $\phi_t$.

\item[ $3 \Rightarrow 4$] 
 Assume that $\int_{\Sigma} \alpha_v(w) \sigma=0$ for every $w$ that is everywhere tangent to $\Sigma$ and satisfies $\mathcal{L}_{{(\op{d}\!i)}^{-1}w}\sigma=0$. 
Decompose $v=\xi_v+\tau_v$, to a vector field $\xi_v$ everywhere $\omega$-orthogonal to $\Sigma$ and a vector field $\tau_v$ everywhere tangent to $\Sigma$,
 as in \eqref{decomp}, so $\alpha_v(\cdot)=\omega(\xi_v,\op{d}\!i(\cdot))+\omega(\tau_v,\op{d}\!i(\cdot))$. By assumption $\alpha_v$ is closed; by Lemma \ref{path} and the step $1 \Rightarrow 2$ above, $\alpha_{\xi_v}$ is closed, hence $\alpha_{\tau_v}$ is closed. When we consider $\tau_v \colon \Sigma \to \op{d}i(\op{T}\! \Sigma) \xrightarrow[]{(\op{d}\!i)^{-1}} \op{T}\!\Sigma$ as a vector field on $\Sigma$, we conclude that $\sigma({(\op{d}\!i)^{-1}}\tau_v,\cdot)$ is closed on $\Sigma$. By Cartan's formula and since $\sigma$ is closed, we get that ${(\op{d}\!i)^{-1}}\tau_v \in \X(\Sigma, \, \sigma)$.
 By the assumption on $v$ and the choice of $\xi_v$, for every $w$ that is everywhere tangent to $\Sigma$ and satisfies $\mathcal{L}_{{(\op{d}\!i)}^{-1}w}\sigma=0$,
$$0=\int_{\Sigma} \alpha_v(w)\sigma=\int_{\Sigma} \sigma({(\op{d}\!i)^{-1}}\tau_v,w)\sigma.$$ In particular, $$\int_{\Sigma} \sigma((\op{d}\!i)^{-1}\tau_v,(\op{d}\!i)^{-1}\tau_v) \sigma=0.$$ Thus (since $\Sigma$ is connected and $\sigma$ is a volume form) $\tau_v=0$, so $\alpha_v=\xi_v$ is everywhere $\omega$-orthogonal to $\Sigma$. 

\item [ $4 \Rightarrow 1$] By Lemma \ref{path}.
\end{itemize}

\end{proof}

\begin{remark} \label{123}
Notice that the steps $1 \Rightarrow 2$, $2 \Rightarrow 3$ and $4 \Rightarrow 1$ hold for every $v  \in \op{T}_{i}C^{\infty}(\Sigma, \, M)$. Only the step $3 \Rightarrow 4$ requires
$v \in \Gamma_{\closed}(i^{*}\op{T}\!M)$.
\end{remark}

\begin{lemma} \label{identifyx}
Let $\Sigma=(\Sigma, \, \sigma) \overset{i}{\hookrightarrow}  (M, \, \omega)$ be an embedded closed connected symplectic submanifold of dimension $2$.
The map $v \mapsto {\op{d}\!i}^{-1}\tau_v$ from $\Gamma_{\closed}(i^{*}\op{T}\!M)$ is onto  $\X(\Sigma, \, \sigma)$ and restricts to a one-to-one and onto map from the subspace $\{\tau_v\, | \, v \in  \Gamma_{\closed}(i^{*}\op{T}\!M)\}$. 
\end{lemma}

\begin{proof}
First, we show that the image is a subset of $\X(\Sigma, \, \sigma)$: by assumption, $v \in \Gamma_{\closed}(i^{*}\op{T}\!M)$; by Lemma \ref{path} and $1 \Rightarrow 2$ in Lemma \ref{vanish}, 
 $\xi_v \in  \Gamma_{\exact}(i^{*}\op{T}\!M) \subset \Gamma_{\closed}(i^{*}\op{T}\!M)$,
hence, $\tau_v =v-\xi_v$ is in the space $\Gamma_{\closed}(i^{*}\op{T}\!M)$. In other words,
\begin{equation}\label{subspace}
\op{d}\!\iota_{\tau_v}\omega|_{\op{d}\!i(\op{T}\!\Sigma)}=0.
\end{equation}
 Therefore, since $i$ is a symplectic embedding, $\op{d}\! \iota_{{\op{d}\!i}^{-1}\tau_v}\sigma=0$. 
Thus, by Cartan's formula, since $\sigma$ is a closed form, $\mathcal{L}_{{\op{d}\!i}^{-1}\tau_v}{\sigma}=0$.  Reversing the argument, we get that for every $\tau \in \X(\Sigma, \, \sigma)$, the vector $\op{d}\!i\tau$ is a vector in $\Gamma_{\closed}(i^{*}\op{T}\!M)$ that is everywhere tangent to $\Sigma$, hence the map is onto.

 By \eqref{subspace}, the space  $\{\tau_v\, | \, v \in  \Gamma_{\closed}(i^{*}\op{T}\!M)\}$ is a subspace of  $\Gamma_{\closed}(i^{*}\op{T}\!M)$. By the above argument, the map $\tau_v \mapsto {\op{d}\!i}^{-1}\tau_v$ on it is onto $\X(\Sigma, \, \sigma)$.

\end{proof}

\begin{corollary}\label{split}
Let $\Sigma=(\Sigma, \, \sigma) \overset{i}{\hookrightarrow}  (M, \, \omega)$ be an embedded closed connected symplectic submanifold of dimension $2$.
Then $$\Gamma_{\closed}(i^{*}\op{T}\!M)= \Gamma_{\exact}(i^{*}\op{T}\!M) \oplus \X(\Sigma, \, \sigma).$$
The splitting gives a convenient space structure on $\Gamma_{\exact}(i^{*}\op{T}\!M) \oplus \X(\Sigma, \, \sigma)$. 
\end{corollary}

\begin{proof}
A vector 
$v \in \Gamma_{\closed}(i^{*}\op{T}\!M)$
decomposes as
$
\xi_v+\tau_v
$
where $\xi_v$ is everywhere $\omega$-orthogonal to $\Sigma$ and $\tau_v$ is everywhere tangent to $\Sigma$.
Such a decomposition exists and is unique, e.g., by Remark \ref{extend} and Corollary \ref{corj}.

By Lemma \ref{vanish}, the space  $\{\xi_v\, | \, v \in  \Gamma_{\closed}(i^{*}\op{T}\!M)\}$ equals  $\Gamma_{\exact}(i^{*}\op{T}\!M)$.
By Lemma \ref{identifyx}, the space $\{\tau_v\, | \, v \in  \Gamma_{\closed}(i^{*}\op{T}\!M)\}$ is identified with $\X(\Sigma, \, \sigma)$. Notice that the maps $v \overset{h_1}\mapsto (\xi_v, \, (\op{d}\! i)^{-1}\tau_v)$ and $(\xi, \, w) \overset{h_2}\mapsto \xi + \op{d}\!i w$ send smooth curves to smooth curves. Moreover, for $c_1 \colon \R \to  \Gamma_{\exact}(i^{*}\op{T}\!M) \oplus \X(\Sigma, \, \sigma)$, if $c_2 \colon \R \to \Gamma_{\closed}(i^{*}\op{T}\!M)$ satisfies ${c_2}'=h_2(c_1)$ then $(h_1(c_2))'=c_1$. 
 
The space $\Gamma_{\closed}(i^{*}\op{T}\!M)$ is convenient since it is the kernel of the continuous map $v \mapsto \alpha_v$ composed on $\alpha \to \op{d}\! \alpha$, from the convenient space $\Gamma(i^{*}\op{T}\!M)$ to the space $\Omega^{1}(\Sigma)$ of $1$-forms on $\Sigma$ and then to $\Omega^{2}(\Sigma)$.
\end{proof}

\begin{proof}[Proof of Proposition \ref{model}]
Given $i \in \mathcal{S}_{\op{e}}(\Sigma, \, \sigma)$, by Weinstein's symplectic tubular neighbourhood theorem (see \S \ref{normal}), the symplectic embedding $i$ can be extended on a neighbourhood $U$ of the zero section in $\op{N}\!\Sigma$ to a symplectic embedding $\Phi_{i} \colon U \to M$. 
By the identification of each fiber $\op{N}_{x}\Sigma$ with $(\op{T}_{i(x)}i(\Sigma))^{\omega}$, and Lemma \ref{vanish}, the elements of $U$ are of the form $(y,\xi({y}))$ where $\xi$ is in 
$\Gamma_{\exact}(i^{*}\op{T}\!M)$. The space $\X(\Sigma, \, \sigma)$ is the Lie algebra of $\Diff(\Sigma, \, \sigma)$, see \cite[43.12]{conv}.
Let $V_{e}$ be 
a chart neighbourhood of the identity map $e \in \Diff (\Sigma, \, \sigma)$ and denote by $$\psi_{e} \colon V_{e} \to \X (\Sigma, \, \sigma)$$ the 
corresponding chart in an atlas on $\Diff (\Sigma, \, \sigma)$. 
Define 
$$W_{i} :=\{ \ell \in \mathcal{S}_{\op{e}}(\Sigma, \, \sigma) \, | \, \ell(x) = \Phi_{i}(b(x), \xi({b(x)})) \text{ for } \xi \in \Gamma_{\exact}(i^{*}\op{T}\!M),  \, b \in V_{e} \text{ s.t.\ }(b(x),\xi({b(x)}))\in U \, \forall x \in \Sigma \},$$ 
$$ \phi_{i} \colon W_i \to \Gamma_{\exact}(i^{*}\op{T}M) \oplus \X(\Sigma, \, \sigma), \, \, \, \, \, \, \,\,\,\,\,\,\,\,\, \phi_{i}(\ell) := (\xi, \psi_{e}(b)).$$
By part (2) of Corollary \ref{split}, $\Gamma_{\exact}(i^{*}\op{T}\!M) \oplus \X(\Sigma, \, \sigma)$ is a convenient space. The set $\{(b(x),\xi({b(x)}))\in U \, \forall x \in \Sigma\}$ is $c^{\infty}$-open in 
$\Gamma_{\exact}(i^{*}\op{T}\!M) \oplus \X(\Sigma, \, \sigma)$. Thus $\phi_i$ is a bijection of $W_i$ onto a $c^{\infty}$-open subset of  $\Gamma_{\exact}(i^{*}\op{T}\!M)  \oplus \X(\Sigma, \, \sigma)$.
The collection $(W_{i},\phi_{i})_{i \in  \mathcal{S}_{\op{e}}(\Sigma, \, \sigma)}$
defines a smooth atlas on $\mathcal{S}_{\op{e}}(\Sigma, \, \sigma)$: the chart changings $\phi_{ik}$ are smooth by smoothness of the exponential map and of each symplectic 
embedding $\Phi_i$.
\end{proof}

 \begin{lemma} \label{vtangent}
 Let $\Sigma=(\Sigma, \, \sigma) \overset{i}{\hookrightarrow}  (M, \, \omega)$ be an embedded closed connected symplectic submanifold of dimension $2$.
\begin{enumerate}
 \item For a section $v \colon \Sigma \to i^{\ast}\op{T}\!M$ that is in $\op{T}_{i}\mathcal{S}_{\op{e}}(\Sigma, \, \sigma)$, the form $$\alpha_{v}=\omega(v,\op{d}\!i(\cdot))$$ is a closed form on $\Sigma$.
 \item Every vector $v \in  \Gamma_{\closed}(i^{*}\op{T}\!M)\}$ can be extended to a vector field $\tilde{v}$ on a neighbourhood of $i(\Sigma)$ in $M$ such that $\mathcal{L}_{\tilde{v}}\omega=0$.
 \end{enumerate}  
    \end{lemma}
    
    \begin{proof}
    \begin{enumerate}
    \item For a vector field  $v$, let $t \to \phi_t$ be the integral curve of $v$ starting at $i$. Since $v$ is tangent to $\mathcal{S}_{\op{e}}(\Sigma, \, \sigma)$, for $w_1, \, w_2$ in $\op{T}\!\Sigma$,
        we get ${\phi_t}^{*}\omega(\op{d}\!iw_1,\op{d}\!iw_2)=\sigma(w_1,w_2)=\omega(\op{d}\!iw_1,\op{d}\!iw_2)$, where $\op{d}\!i \colon \op{T}\Sigma \to \op{T}(i (\Sigma))$.
    Hence  on $i(\Sigma)$, $$\mathcal{L}_{v}{\omega}=\lim_{t \to 0} \frac{{\phi_t}^{*}\omega-\omega}{t}=0.$$
    By Cartan's formula, this implies $\op{d}\!\iota_{v}\omega=0$ as a form on $i(\Sigma)$, i.e., $\alpha_v$ is a closed form on $\Sigma$.
  
    \item By Cartan's formula and the fact that $\omega$ is a closed form, we need to extend $v$ to $\tilde{v}$ on a neighbourhood of $i(\Sigma)$ in $M$ such that the $1$-form $\omega(\tilde{v},\cdot)$ is a closed form.
    By the decomposition \eqref{decomp} and Lemma \ref{path}, it is enough to extend $\tau_v \in \Gamma_{\closed}(i^{*}\op{T}\!M)$ that is everywhere tangent to $\Sigma$ to such a vector.
  The closed form $\iota_{(\op{d}\!i)^{-1}\tau_v}\sigma$ on the zero section of $(\op{N}\! \Sigma, \, \omega_{\op{N}\!\Sigma})$ pulls back (through the projection of  $\op{N}\!\Sigma$ to the zero section) to a
closed one-form on $\op{N}\!\Sigma$ that is consistent with $\iota_{(\op{d}\!i)^{-1}\tau_v}\sigma$ on the zero section and zero on directions $\omega_{\op{N}\!\Sigma}$-orthogonal to the zero section. 
  By Weinstein's symplectic tubular neighbourhood theorem, the push forward of this form via the symplectic embedding
$\Phi_i \colon (U, \, \omega_{\op{N}\!\Sigma}) \to (M, \, \omega)$ of \eqref{eqnormal} 
is a closed one-form 
$\widetilde{\alpha}_{\tau_v}$  on a neighbourhood of $i(\Sigma)$ in $M$ that is consistent with $\iota_{\tau_v}\omega$ on vectors tangent to $i(\Sigma)$.  Define $\widetilde{\tau_v}$ to be the vector field such that $\widetilde{\alpha}_{\tau_v}(\cdot)=\omega(\widetilde{\tau_v},\cdot)$. The vector $\widetilde{\tau_v}$ is well defined since $\omega$ is non-degenerate.
    
    \end{enumerate}
     \end{proof}

Denote by $\omega_{\mathcal{S}_{\op{e}}(\Sigma, \, \sigma)}$ the pullback (through inclusion) of $\omega_{\mathcal{S}_{\op{i}}(\Sigma)}$ to $\mathcal{S}_{\op{e}}(\Sigma, \, \sigma)$, 
and by ${{\omega^{D}}}_{\mathcal{S}_{\op{e}}(\Sigma, \, \sigma)}$ the pullback of $ {\omega^{D}}_{(\Sigma, \, \sigma)}$  to $\mathcal{S}_{\op{e}}(\Sigma, \, \sigma)$.

\begin{proposition} \label{sympl3}
The $2$\--form  ${{\omega^{D}}}_{\mathcal{S}_{\op{e}}(\Sigma, \, \sigma)}$ on $\mathcal{S}_{\op{e}}(\Sigma, \, \sigma)$ is closed and weakly non-degenerate. 
\end{proposition}
\begin{proof}
The form is closed as the restriction of the closed form $ {\omega^{D}}_{(\Sigma, \, \sigma)}$
(see proposition \ref{thclonon2}) to the manifold $\mathcal{S}_{\op{e}}(\Sigma, \, \sigma)$ (see proposition \ref{model}). We need to show that it is weakly non-degenerate.
\begin{enumerate}
\item For $0\neq \tau_v \in \op{T}_{i}\mathcal{S}_{\op{e}}(\Sigma, \, \sigma)$ that is everywhere tangent to $\Sigma$, 
$$\omega^D(\tau_v,\tau_v)=\int_{\Sigma}\omega(\tau_v,\tau_v)\sigma=\int_{\Sigma}\sigma(\op{d}\!i^{-1}\tau_v,\op{d}\!i^{-1}\tau_v)\sigma \neq 0.$$
The last inequality is since $\Sigma$ is connected, $\sigma$ is a volume form, $\tau_v \neq 0$ and $\op{d}\!i \colon \op{T}\!\Sigma \to \op{T}\!i(\Sigma)$ is an isomorphism.

\item Suppose that $w \in \op{T}_{i}({\mathcal{S}_{\op{e}}(\Sigma, \, \sigma)})$ is not tangent to $\Sigma$ at $x \in \Sigma$.
By Lemma \ref{path}, $w=\xi_w+\tau_w$, where $\xi_w=i^{*}\xi$ with $\xi$ a Hamiltonian vector field on $M$ and  $\tau_w$ everywhere tangent to $\Sigma$.
In particular, $\xi_{w}(x) \neq 0$.
Let $w_1$ be a vector in $(i^{\ast} (\op{T}\! M))_x$ such that 
 $$\omega(\xi_{w} (x),\,w_1) > 0,$$ 
 and $w_1$ is symplectically orthogonal to $\op{d}\!i_x ( \op{T}_x \Sigma )$. (For example, $w_1=J \xi_{w}(x)$ for an almost complex structure $J$ that is $\omega$-compatible. See  part (2) of Claim \ref{jclaim} and Remark \ref{extend}.) 
  
   Now extend $w_1$ to a section $w_1   
   \colon \Sigma \rightarrow i^{\ast}(\op{T}\! M)$ in  $\op{T}_{i} {\mathcal{S}_{\op{e}}(\Sigma, \, \sigma)}$
  such that $\omega( \xi_{w}( y), \, w_1(y) ) > 0$ and $w_1(y)$ is symplectically orthogonal to $\op{d}\!i_y ( \op{T}_y \Sigma )$ for $y$ in a small
  neighborhood of $x$, and vanishing outside it.  By Lemma \ref{path}, $w_1=i^{*}W_H$ with $W_H$ a Hamiltonian vector field on $M$. By Lemma \ref{vtangent}, $w \in \Gamma_{\closed}(i^{*}\op{T}\!M)$, hence (see Lemma \ref{identifyx}), $\tau_w$ is an everywhere tangent to $\Sigma$ vector that satisfies  $\mathcal{L}_{{\op{d}\!i}^{-1}\tau_w}{\sigma}=0$.
     Then
   $$ { ({{\omega^{D}}}_{\mathcal{S}_{\op{e}}(\Sigma, \, \sigma)})}_{i}(w,w_1)={ ({{\omega^{D}}}_{\mathcal{S}_{\op{e}}(\Sigma, \, \sigma)})}_{i}(\xi_w,w_1)=
  \int_{\Sigma} \omega(\xi_{w},w_1) \sigma \neq 0,$$
  where the first equality follows from $1 \Rightarrow 3$ in Lemma \ref{vanish} and the last inequality follows from the choice of $w_1$. 

\end{enumerate}

\end{proof}

 \section{The quotient of  $\mathcal{S}_{\op{e}}(\Sigma, \, \sigma)$ by symplectic reparametrizations} \label{sec3}

\subsection*{The forms $\omega_{\mathcal{S}_{\op{e}}(\Sigma, \, \sigma)}$ and   ${{\omega^{D}}}_{\mathcal{S}_{\op{e}}(\Sigma, \, \sigma)}$   coincide in exact directions}

\begin{claim}
For tangent vectors $v_1, \, v_2 \in \op{T}_{f}\mathcal{S}_{\op{e}}(\Sigma, \, \sigma)$, the integrand $\iota_{( (v_1,0) \wedge (v_2,0) )} \op{ev}^{\ast} ( \omega \wedge \omega )$ equals 
  $$
 2 \, \omega(v_1,\, v_2)\, \omega(\op{d}\!f(\cdot),\, \op{d}\!f(\cdot)) + 2 \omega (v_1, \, \op{d}\!f(\cdot)) \wedge \omega(v_2, \, \op{d}\!f(\cdot))= 2 \, \omega(v_1,\, v_2)\, f^{*}\omega(\cdot, \, \cdot) + 2 \omega (v_1, \, \op{d}\!f(\cdot)) \wedge \omega(v_2, \, \op{d}\!f(\cdot)),
  $$
which, since $f \in \mathcal{S}_{\op{e}}(\Sigma, \, \sigma)$, equals 
\begin{equation} \label{omom}
2 \, \omega(v_1,\, v_2)\, \sigma(\cdot, \, \cdot) + 2 \omega (v_1, \, \op{d}\!f(\cdot)) \wedge \omega(v_2, \, \op{d}\!f(\cdot)).
\end{equation}
\end{claim}

\begin{lemma} \label{exact}
For $u \in  \op{T}_{f}\mathcal{S}_{\op{e}}(\Sigma, \, \sigma)$ such that $\omega(u,\op{d}\!f(\cdot))$ on $\Sigma$ is exact, we get 
that 
$$\iota_{u}(\omega_{\mathcal{S}_{\op{e}}(\Sigma, \, \sigma)})_f = 2 \, \int_{\Sigma} \omega(u,\cdot) \, \sigma = 2 \iota_{u} ({{\omega^{D}}}_{\mathcal{S}_{\op{e}}(\Sigma, \, \sigma)})_f,$$
i.e., the two forms coincide in exact directions, up to multiplication by a constant. 
\end{lemma}

\begin{proof}
Indeed, since $\omega(u, \op{d}\!f(\cdot))=\op{d}\!h$ for a function $h \colon \Sigma \to \R$, and for every $v \in  \op{T}_{f}\mathcal{S}_{\op{e}}(\Sigma, \, \sigma)$ the form $\alpha_v$ on $\Sigma$ is closed (see Lemma \ref{vtangent}), the integration of the second term in \eqref{omom} along $f(\Sigma)$ vanishes:
$$\int_{\Sigma}   \omega (v, \, \op{d}\!f(\cdot)) \wedge \omega(u, \, \op{d}\!f(\cdot))= 
\int_{\Sigma} \alpha_v \wedge \op{d}\!h =\int_{\Sigma}  \alpha_v \wedge \op{d}\!h - \op{d}\!\alpha_v \wedge h= \int_{\Sigma}\op{d}(\alpha_v \wedge h)=\int_{\partial (\Sigma)} \alpha_v \wedge h=0.$$ 
\end{proof}
Note that (in the second equality) we used the fact that $\alpha_v$ on $\Sigma$ is closed (by Lemma \ref{vtangent}), so the argument works in the space of symplectic embeddings and not for arbitrary embeddings.

\subsection*{Directions of degeneracy for the form $\omega_{\mathcal{S}_{\op{e}}(\Sigma, \, \sigma)}$}

\begin{lemma} \label{tangent}
Let $i \in \mathcal{S}_{\op{e}}(\Sigma, \, \sigma)$.
\begin{enumerate}
\item  If $w \in \op{T}_{i}\mathcal{S}_{\op{e}}(\Sigma, \, \sigma)$ is not everywhere tangent to $\Sigma$, then  
 there exists $w_1  \in \op{T}_{i}\mathcal{S}_{\op{e}}(\Sigma, \, \sigma)$ satisfying $w_1=i^{*}W_H$ with $W_H$ a Hamiltonian vector field on $M$
 such that 
 ${ ({{\omega^{D}}}_{\mathcal{S}_{\op{e}}(\Sigma, \, \sigma)})}_i(w,w_1) \neq 0$.
\item
If $\tau \in \op{T}_{i}\mathcal{S}_{\op{e}}(\Sigma, \, \sigma)$ is everywhere tangent to $\Sigma$, then  
$\iota_{\tau} \,   (\omega_{\mathcal{S}_{\op{e}}(\Sigma, \, \sigma)})_i =0$.
\item If $w \in \op{T}_{i}\mathcal{S}_{\op{e}}(\Sigma, \, \sigma)$ is not everywhere tangent to $\Sigma$, then  
 there exists $w_1  \in \op{T}_{i}\mathcal{S}_{\op{e}}(\Sigma, \, \sigma)$  satisfying $w_1=i^{*}W_H$ with $W_H$ a Hamiltonian vector field on $M$
 such that 
 $(\omega_{\mathcal{S}_{\op{e}}(\Sigma, \, \sigma)})_i(w,w_1) \neq 0$.
\end{enumerate}
\end{lemma}

\begin{proof}
\begin{enumerate}
\item This is shown in the proof of Proposition \ref{sympl3}.
\item Suppose that $\tau$ is everywhere tangent to $\Sigma$.
  Lift $\tau$ to a vector field $\ell=(\tau,0)$; let $\tau_2 \in \op{T}_i \mathcal{S}_{\op{e}}(\Sigma, \, \sigma)$ and $\ell_2$ a lifting of $\tau_2$. We show that the integrand $\iota_{\ell \wedge \ell_2} \op{ev}^{*} (\omega \wedge \omega)$ vanishes when restricted to $\op{T}(\{i\} \times \Sigma)$. 
 Indeed, for $z_1,z_2 \in \op{T}_x(\{i\} \times \Sigma),$ by definition and Lemma \ref{trivial:lem} below,
$$ \iota_{\ell \wedge \ell_2} \op{ev}^{*}(\omega \wedge \omega)_x(z_1,z_2) = \iota_{\tau \wedge \op{d} \op{ev} (\ell_2)}(\omega \wedge \omega)_{i(x)}(\op{d}\! f (z_1), \op{d}\! i(z_2)).$$ So it is enough to show that 
$$ \iota_{\tau \wedge \op{d} \op{ev} (\ell_2)}(\omega \wedge \omega) | _{\op{d}\!i(\op{T}_{x} \Sigma)}$$ vanishes. This follows from Lemma \ref{lemma3}, since, by assumption, $\tau(x) \in \op{d}\!i_{x}(\op{T}_{x}\Sigma)$ and $\, \op{d}\!i_{x}(\op{T}_{x}\Sigma) \subset \op{T}_{i(x)}M$ is a two-dimensional subspace.
\item By item (2) and Lemma  \ref{path} it is enough to prove item (3) with the assumption that $w=i^{*}V_H$ with $V_H$ a Hamiltonian vector field on $M$. By Lemma \ref{exact},  this case follows from item (1). 
\end{enumerate}
\end{proof}

\begin{remark}\label{ro}
By the same argument we get that also on ${\mathcal{S}_{\op{i}}(\Sigma)}$,
we have 
 $\omega_{\mathcal{S}_{\op{i}}(\Sigma)} (\tau, \,\cdot ) = 0$ at $f$ iff
$\tau$  is tangent to $f(\Sigma)$ at every $x \in \Sigma$. 
  See \cite[Thm 1]{ckp}.
  \end{remark}

\begin{lemma} \label{trivial:lem}
For $(\nu,\,v_{\Sigma}) \in \op{T}(\op{C}^{\infty}(\Sigma,M) \times \Sigma)$,
\begin{eqnarray}
  \op{d}(\op{ev}) _{( f, \,x )} (\nu,\,v_{\Sigma}) =  \nu_{f} (x) + \op{d}\! f_{ x } (
v_{\Sigma})  \label{De} \nonumber
\end{eqnarray}
In particular, 
\begin{equation} \label{eq1}
\op{d}(\op{ev})|_{\op{T} ( \{f\} \times \Sigma )} =\op{d} \! f,
\end{equation}
and 
\begin{equation} \label{eq2}
\op{d}(\op{ev})_{(f, \, x)}(\nu,0)=\nu_f (x).
\end{equation}
\end{lemma}

\begin{lemma} \label{lemma3}
   Let $W$ be a vector space, and let $\alpha$ be a
  $4$\--form: $\alpha \colon \bigwedge^4 W \rightarrow \R.$ 
  Let $V \subset W$ be a
  subspace of dimension $\leq 2$. Then
 $\Big(\iota_{( v \wedge w )} \alpha\Big) |_V = 0$
  for all $v \in V$, $w \in W$.  
 \end{lemma}
 \begin{proof}
This is since  any three vectors in
$V$ are linearly
  dependent.
  \end{proof}

As a result of Lemma \ref{tangent} and Lemma \ref{path}, we get an extension of Lemma \ref{exact}: 
\begin{corollary} \label{coin}
For every $u \in  \op{T}_{f}\mathcal{S}_{\op{e}}(\Sigma, \, \sigma)$, there is  $w \in  \op{T}_{f}\mathcal{S}_{\op{e}}(\Sigma, \, \sigma)$ that is everywhere tangent to $f$, such that  $\omega(u+w,\op{d}\! f(\cdot))$ on $\Sigma$ is exact and 
$$\iota_{u}(\omega_{\mathcal{S}_{\op{e}}(\Sigma, \, \sigma)})_f = \iota_{u+w}(\omega_{\mathcal{S}_{\op{e}}(\Sigma, \, \sigma)})_f=2 \, \int_{\Sigma} \omega(u+w,\cdot) \, \sigma = 2 \iota_{u+w} ({{\omega^{D}}}_{\mathcal{S}_{\op{e}}(\Sigma, \, \sigma)})_f,$$  $w$ equals zero if $\omega(u,\op{d}\!f(\cdot))$ on $\Sigma$ is exact.  
\end{corollary}

\subsection*{A symplectic form on the quotient by symplectic reparametrizations}

The group $\Diff(\Sigma,\, \sigma)$ of symplectomorphisms of $(\Sigma,\, \sigma)$ is a Lie group in the Convenient Setup, its Lie algebra is $\X(\Sigma, \, \sigma)$  \cite[43.12]{conv}.
The Lie group $\Diff(\Sigma, \, \sigma)$ acts freely on  $\mathcal{S}_{\op{e}}(\Sigma, \, \sigma)$ on the right by
\begin{equation}\label{saction}
\psi . i = i \circ \psi^{-1}.
\end{equation}
Denote the quotient map
\begin{equation} \label{qmap}
 q \colon \mathcal{S}_{\op{e}}(\Sigma, \, \sigma) \to \mathcal{S}_{\op{e}}(\Sigma, \, \sigma) / \Diff(\Sigma, \, \sigma).
 \end{equation}

Denote by $\Ham(M, \, \omega)$ the Lie group of Hamiltonian symplectomorphisms of $(M, \, \omega)$; its Lie algebra $\g$ is the space of Hamiltonian vector fields \cite[43.12, 43.13]{conv}. (A vector field $X$ on $M$ is Hamiltonian if the form $\iota_{X}\omega$ is exact; a symplectomorphism of $(M, \, \omega)$ is Hamiltonian if it is the time one flow of a time dependent Hamiltonian vector field.) 
The group $\Ham(M, \, \omega) \subset \Diff(M, \, \omega)$  acts freely on  $\mathcal{S}_{\op{e}}(\Sigma, \, \sigma)$ on the left by
$$\phi . i = \phi \circ i.$$
The action descends to the quotient $\mathcal{S}_{\op{e}}(\Sigma, \, \sigma) / \Diff(\Sigma, \, \sigma)$ as
$$
\phi.[i]=[\phi \circ i].
$$

Denote by $\Emb(\Sigma, \, M)$ the space of embeddings of $\Sigma$ into $M$. (An open set in the manifold $C^{\infty}(\Sigma,\,M)$.) The Lie group $\Dif(\Sigma)$ \cite[43.1]{conv} of reparametrizations of $\Sigma$ acts freely on  $\Emb(\Sigma, \, M)$ on the right by
$$
\psi . i = i \circ \psi^{-1}.
$$
The group $\Ham(M, \, \omega)$ acts on  $\Emb(\Sigma, \, M)$, (and its quotient by  $\Dif(\Sigma)$), by $\phi . i = \phi \circ i.$

\begin{lemma} \label{moser}
A path-connected component of the space $\mathcal{S}_{\op{e}}(\Sigma, \, \sigma)$ modulo the action of $\Diff(\Sigma, \, \sigma)$ is identified with a path-connected component of the space $\Emb(\Sigma, \, M)$ modulo the action of $\Dif(\Sigma)$ as follows. 
Let $f \colon \Sigma \to M$ be an embedding that is connected to an element $i \in \mathcal{S}_{\op{e}}(\Sigma, \, \sigma)$ through a path in $\Emb(\Sigma, \, M)$. There exists a $\Sigma$-reparametrization $\varphi_f$ such that $f \circ \varphi_f \in  \mathcal{S}_{\op{e}}(\Sigma, \, \sigma)$; the $\Sigma$-reparamertization $\varphi_f$ is unique up to symplectic $\Sigma$-reparametrizations. The map $[f] \mapsto [f \circ \varphi_f]$ from  the quotient of the path-connected component of $i \in \mathcal{S}_{\op{e}}(\Sigma, \, \sigma)$ in $\Emb(\Sigma, \, M)$ by $\Dif(\Sigma)$ to $\mathcal{S}_{\op{e}}(\Sigma, \, \sigma)/ \Diff(\Sigma, \, \sigma)$ is well defined, one-to-one and its image is the quotient of the path-connected component of $i$ in $\mathcal{S}_{\op{e}}(\Sigma, \, \sigma)$ by the action of  $\Diff(\Sigma, \, \sigma)$.

The map $f \mapsto f \circ \varphi_f$ sends a $\Ham(M,\, \omega)$-orbit  in  $\Emb(\Sigma, \, M)$ to a $\Ham(M, \, \omega)$-orbit in $\mathcal{S}_{\op{e}}(\Sigma, \, \sigma)$  modulo symplectic reparametrizations. 
\end{lemma}

\begin{proof}
 Let $\{f_t\}_{0 \leq t \leq 1}$ be a path in $\Emb(\Sigma, \, M)$ with $f_0 =i \in \mathcal{S}_{\op{e}}(\Sigma, \, \sigma)$ and $f_1=f$. 
Set $$\omega_t={f_t}^{*}\omega.$$ By definition $\omega_t$ is closed. By the Homotopy invariance of de Rham cohomology, $[\omega_t]=[\omega_0]$ for all $t$. Therefore 
\begin{equation} \label{eqint}
\int_{\Sigma}{f_t}^{*}\omega-\int_{\Sigma}{f_0}^{*}\omega=\int_{\Sigma} \op{d}\!\eta=\int_{\partial\Sigma}\eta=0,
\end{equation}
 where the last equality is since $\Sigma$ is closed and the equality before it is by Stoke's theorem.  Since ${f_0}^{*}\omega=i^{*}\omega=\sigma$, equation \eqref{eqint} implies that for all $t$, the integral $\int_{\Sigma}{\omega_t}=\int_{\Sigma}{f_t}^{*}\omega \neq 0$ therefore, the $2$-form $\omega_t$ is a non-vanishing volume form on the $2$-dimensional manifold $\Sigma$ hence non-degenerate. We can then apply Moser's theorem \cite{moser} to get an isotopy $\varphi \colon \Sigma \times \R \to \Sigma$ such that 
 \begin{equation} \label{eqft}
 (f_t \circ \varphi_t)^{*}\omega={\varphi_t}^{*}\omega_t=\omega_0=\sigma \text{    for }0 \leq t \leq 1.
 \end{equation}
  Set $\varphi_f:=\varphi_1$.  

Notice that if $\psi_1, \, \psi_2 \in \Dif(\Sigma)$ are such that $(f \circ \psi_i)^{*}\omega=\sigma$, then $$({\psi_2^{-1}} \circ \psi_1)^{*}\sigma= ({\psi_2^{-1}} \circ \psi_1)^{*}(f \circ \psi_2)^{*}\omega=(f \circ \psi_2 \circ {\psi_2^{-1}} \circ \psi_1)^{*}\omega=(f \circ \psi_1)^{*}\omega=\sigma,$$
so $[f \circ \psi_1]=[f \circ \psi_2]$ in  $\mathcal{S}_{\op{e}}(\Sigma, \, \sigma)$ modulo $\Diff(\Sigma, \, \sigma)$. Therefore the class in  $\mathcal{S}_{\op{e}}(\Sigma, \, \sigma)$ modulo $\Diff(\Sigma, \, \sigma)$ we associated to $f$ is independent of the choice of path and isotopy. Similarly, it is independent of the choice of the representative of $[f]$: if $g \colon \Sigma \to M$ and $\alpha, \, \beta \in \Dif(\Sigma)$ are such that $f=g \circ \alpha$ and $(g \circ \beta)^{*}\omega=\sigma$, then $$(\beta^{-1} \circ \alpha \circ \varphi_f)^{*}\sigma=(\beta^{-1} \circ \alpha \circ \varphi_f)^{*}(g \circ \beta)^{*}\omega=(g \circ \beta \circ \beta^{-1} \circ \alpha \circ \varphi_f)^{*}\omega=(f \circ \varphi_f)^{*}\omega=\sigma,$$ so $[f \circ \varphi_f]=[g \circ  \beta]$. Therefore
the assignment  $[f] \mapsto [f \circ \varphi_f]$ is well defined. It is also one-to-one:
if $g$ and $f$ are not in the same class in $\Emb(\Sigma, \, M)$ modulo $\Dif(\Sigma)$, the associated $g \circ \varphi_g$ and $f \circ \varphi_f$  cannot be equal up to symplectic reparametrization of $\Sigma$. By the construction of $\varphi_f$ (see \eqref{eqft}), the map $f \mapsto f \circ \varphi_f$ sends paths starting from $i$ to paths starting from $i$ composed on a symplectic reparametrization of $(\Sigma,\, \sigma)$, hence the image of $[f] \mapsto [f \circ \varphi_f]$ is contained in the quotient of the path-connected component of $i$ in $\mathcal{S}_{\op{e}}(\Sigma, \, \sigma)$ modulo $\Diff(\Sigma, \, \sigma)$; it  is onto this quotient since for every symplectic embedding $h$ that is path-connected to $i$ in  $\mathcal{S}_{\op{e}}(\Sigma, \, \sigma)$, the class  of $h$ in  $C$ modulo  $\Dif(\Sigma)$ is sent to the class of $h$ in  $\mathcal{S}_{\op{e}}(\Sigma, \, \sigma)$ modulo $\Diff(\Sigma, \, \sigma)$.

Since the $\Ham(M,\, \omega)$-actions are from the left while the actions of $\Dif(\Sigma)$ and $\Diff(\Sigma,\sigma)$ are from the right, the map $f \mapsto f \circ \varphi_f$ sends an $\Ham(M,\, \omega)$-orbit  in  $\Emb(\Sigma, \, M)$ to an $\Ham(M, \, \omega)$-orbit in $\mathcal{S}_{\op{e}}(\Sigma, \, \sigma)$  modulo symplectic reparametrizations. For $\phi \in \Ham(M, \, \omega)$, the embedding $h = \phi \circ f$ is connected to $\phi \circ i \in  \mathcal{S}_{\op{e}}(\Sigma, \, \sigma)$ through a path in  $\Emb(\Sigma, \, M)$, hence there exists $\varphi_h \in \Dif(\Sigma)$ such that $h \circ \varphi_h \in \mathcal{S}_{\op{e}}(\Sigma, \, \sigma)$.
Since $$({\varphi_f}^{-1} \circ \varphi_h)^{*}\sigma=({\varphi_f}^{-1} \circ \varphi_h)^{*}(f \circ \varphi_f)^{*}\omega=(f \circ \varphi_h)^{*}\omega=(\phi^{-1}\circ h \circ \varphi_h)^{*}\omega=(h \circ \varphi_h)^{*}{\phi^{-1}}^{*}\omega=(h \circ \varphi_h)^{*}\omega=\sigma,$$
the symplectic embeddings $h \circ \varphi_h=(\phi \circ f) \circ \varphi_h$ and $\phi \circ (f \circ \varphi_f)$ are in the same class of  $\mathcal{S}_{\op{e}}(\Sigma, \, \sigma)$ modulo $\Diff(\Sigma, \, \sigma)$.

\end{proof}

\begin{noTitle}\label{smooth}
A neighbourhood of an element $[i]$ in $\Emb(\Sigma, \, M)$ modulo $\Dif(\Sigma)$ is identified with a neighbourhood of the zero section in the space of sections of the Normal bundle $\op{N}\!\Sigma$. The identification is as follows: choose a Riemannian metric on $M$ such that $\op{N}\! \Sigma$ is the orthogonal complement of $\op{T}\!\Sigma=\op{T}\!(i \Sigma)$ in $\op{T}\!M|_{i(\Sigma)}$; the exponential map with respect to that metric sends a neighbourhood of the zero section to a neighbourhood  of the submanifold $i(\Sigma)$. This defines an atlas on  $\Emb(\Sigma,\,M) / \Dif(\Sigma)$, all translation functions are smooth by the smoothness of  the exponential map.  
Hence, by Lemma \ref{moser}, we get a smooth structure on   $\mathcal{S}_{\op{e}}(\Sigma, \, \sigma)$ modulo $\Diff(\Sigma, \, \sigma)$.

The map $\exp$ sends an orbit of the linearized action of $\Ham(M, \, \omega)$ on $\op{N} \! {\Sigma}$ onto an orbit of $\Ham(M, \, \omega)$ in   $\Emb(\Sigma,\,M)$ modulo reparametrizations.
\end{noTitle}

\begin{corollary}\label{hampath}
A path of symplectic embeddings $(\Sigma, \, \sigma) \to (M, \, \omega)$ through $f$ can be written, up to symplectic reparametrizations of $(\Sigma, \, \sigma)$,
as $\Psi_t \circ f$, where  
$\Psi_t$
is a path in $\Ham(M, \, \omega)$. 
\end{corollary}
\begin{proof}
Note that for a time dependent vector field $v_t \in \op{T}_{i}\mathcal{S}_{\op{e}}(\Sigma, \, \sigma) \times \R$, the decomposition \eqref{decomp} is smooth with respect to the time parameter.
Thus the corollary follows from Lemma \ref{path} and Lemma \ref{identifyx}.
 \end{proof}


\begin{definition}\label{leedef} \cite[Definition 10]{lee}.
Let $(X, \,\Omega)$ be a weakly symplectic manifold. Let $G \curvearrowright X$ be a 
free action of a Lie group $G$ on $X$, such that 
${g}^{*} \Omega = \Omega$ for all $g \in G$ (${g} \colon X \to X$ denotes the action $x \to g.x$). The collection of subspaces 
$$\mathcal{D}_x := \{v \in \op{T}_{x} X \,|\, \Omega_{x} (v, \xi_{X} (x))=0 \, \, \forall \xi \in \mathfrak{g}\}$$ 
for $x \in X$ defines a distribution $\mathcal{D}$ on $X$ . Let $i_{N}  \colon  N \hookrightarrow X$ be a maximal integral manifold of $\mathcal{D}$ 
and let $q \colon X \to X /G$ denote the projection to the orbit space. Suppose that
the quotient induces a smooth structure on $q(N)$, in which there exists a unique weak symplectic structure $\Omega_{\red}$ on $q(N)$ such that $(q|_{N})^{*}\Omega_{\red}=i^{*}_{N}\Omega$.
Then the weakly symplectic manifold $(q(N), \,\Omega_{\red})$ will be called a \emph{reduction} or \emph{symplectic quotient} of $(X,\,\Omega)$ with respect to the $G$-action. 
\end{definition}

A \emph{distribution} on a smooth manifold $M$ assigns to each point $x \in M$ a $c^{\infty}$-closed 
subspace $\mathcal{D}_x$ of $\op{T}_{x}M$. (The $c^{\infty}$-topology on a locally convex space $E$ is the finest topology for which all smooth curves $c \colon \R \to E$ are continuous.) If $\mathcal{D} = \{\mathcal{D}_x \}$ is a distribution on a manifold $M$ and $i \colon N \hookrightarrow M$ is the 
inclusion map of a path-connected submanifold $N$ of $M$, then $N$ is called an integral manifold of $\mathcal{D}$ if 
$\op{d}_{i} (\op{T}_{x} N ) = \mathcal{D}_{i(x)}$ for all $x \in N$. An integral manifold of $\mathcal{D}$ is called maximal if it is not properly 
contained in any other integral manifold. (Note that since the local model is a locally convex vector space, a manifold is path-connected if and only if it is connected.)

\begin{noTitle}
The motivation for Definition \ref{leedef} comes from the standard reduction of a finite dimensional symplectic manifold 
$(X,\, \Omega)$ with respect to a Hamiltonian $G$-action with a moment map $\phi$. For a regular 
value $r$ of $\phi$,  the tangent space at $p$ to the level surface $\phi^{-1} (r)$ is equal to the set ${\mathcal{D}}_p$ of all 
vectors $v \in  \op{T}_{p}X$  satisfying $\Omega(v, \xi_{X}(p)) = 0$ for all $\xi \in \mathfrak{g}$. The subspaces ${\mathcal{D}}_p$ 
give a distribution $\mathcal{D}$ 
on $X$,
defined even in the absence of a moment map. If $G \curvearrowright X $ is a free symplectic action, then this distribution can be taken as the starting 
point of the Òoptimal reduction methodÓ of Ortega and Ratiu \cite{ratiu}. 
\end{noTitle}

\begin{theorem} \label{lee}
Consider   $(\mathcal{S}_{\op{e}}(\Sigma, \, \sigma), \,{{\omega^{D}}}_{\mathcal{S}_{\op{e}}(\Sigma, \, \sigma)})$ 
with the action \eqref{saction} of  $\Diff(\Sigma, \, \sigma)$.
The $\Ham(M, \, \omega)$-orbit $\mathcal{N}$ through $i \in  \mathcal{S}_{\op{e}}(\Sigma, \, \sigma)$ is a maximal  integral manifold of the distribution $\mathcal{D}$. The restriction of  ${{\omega^{D}}}_{\mathcal{S}_{\op{e}}(\Sigma, \, \sigma)}$ to $\mathcal{N}$ descends to a weak symplectic structure $\omega^D_{\red}$ on the image $\mathcal{O}:=q(\mathcal{N})$ in the orbit space under the projection $ q \colon \mathcal{S}_{\op{e}}(\Sigma, \, \sigma) \to \mathcal{S}_{\op{e}}(\Sigma, \, \sigma) / \Diff(\Sigma, \, \sigma).$ Thus the symplectic space $(\mathcal{O}, \, \omega^D_{\red})$ is a reduction of $\mathcal{S}_{\op{e}}(\Sigma, \, \sigma)$
with respect to the $\Diff(\Sigma, \, \sigma)$-action. 
\end{theorem}

\begin{proof}
By Proposition \ref{model} and Proposition \ref{sympl3},   $(\mathcal{S}_{\op{e}}(\Sigma, \, \sigma), \,{{\omega^{D}}}_{\mathcal{S}_{\op{e}}(\Sigma, \, \sigma)})$  is a weakly symplectic manifold.
Let  $i \in  \mathcal{S}_{\op{e}}(\Sigma, \, \sigma)$. 
By definition, 
$$D_i=\{v \in  \op{T}_{i} {\mathcal{S}_{\op{e}}(\Sigma, \, \sigma)}\, | \, \omega^{D}_{i}(v, \xi_{ \mathcal{S}_{\op{e}}(\Sigma, \, \sigma)}(i))=0 \, \forall \xi \in \X(\Sigma, \, \sigma) \}.$$
By $1 \Leftrightarrow 3$ in Lemma  \ref{vanish}, and the fact that $\op{T}_{i} {\mathcal{S}_{\op{e}}(\Sigma, \, \sigma)} \subset \Gamma_{\closed}(i^{*}\op{T}\!M)$ (by Lemma \ref{vtangent}),
$$ \op{T}_{i}(\Ham(M, \, \omega) . i)=\{v \in  \op{T}_{i} {\mathcal{S}_{\op{e}}(\Sigma, \, \sigma)}\, | \, \omega^{D}_{i}(v, \xi)=0 \, \forall \xi \in \Gamma_{\closed}(i^{*}\op{T}\!M)\text{ that is everywhere tangent to }\Sigma\}.$$  
By Lemma \ref{identifyx},  the space of vector fields in $\Gamma_{\closed}(i^{*}\op{T}\!M)$ that are everywhere tangent to $\Sigma$ is identified with the space $\X(\Sigma, \, \sigma)$.
So $\Ham(M, \, \omega)$-orbits are integral manifolds of $\mathcal{D}$. 

To see that maximal, assume that the $\Ham(M, \, \omega)$-orbit $\mathcal{N}$ is properly contained in a (path-connected) integral manifold $\tilde{N}$. By Lemma \ref{hampath} and Corollary \ref{split}, a path in $\tilde{N}$ from an element $f \in \mathcal{N}$ to an element in $\tilde{N} \smallsetminus N$ can be written
as $\Psi_t \circ f \circ \alpha_t$, where $\Psi_t$
is a path in $\Ham(M, \, \omega)$ and $\alpha_t \colon \Sigma \to \Sigma$ are in $\Diff(\Sigma,\,\sigma)$; moreover, the generating vector field of the path $v_t$ decomposes uniquely as $\xi_{v_t} + \tau_{v_t}$, where $\xi_{v_t}$ is symplectically orthogonal to $\Sigma$ and $\tau_{v_t}$ is everywhere tangent to $\Sigma$, and $\xi_{v_t}$ is the vector field generating $\Psi_t$ and $\tau_{v_t}$ is generating $\alpha_t$. Since   $\op{d}_{i_{\tilde{N}}} (\op{T}_{x} \tilde{N} ) = \mathcal{D}_{i_{\tilde{N}}(x)}$ for all $x \in \tilde{N}$ (where $i_{\tilde{N}}$ is the inclusion map of $\tilde{N}$ into $M$), and by $3 \Leftrightarrow 4$ in Lemma \ref{vanish}, for every $v_t \in \op{d}_{i_{\tilde{N}}} (\op{T}_{x} \tilde{N} )$, the vector $v_t=\xi_{v_t}$ and $\tau_{v_t}=0$. Hence $\alpha_t=\id$ for every $t$, and we get a path in $\Ham(M, \, \omega)$ connecting an element in a $\Ham(M, \, \omega)$-orbit with an element outside the orbit: a contradiction. 
  
Consider the inclusion-quotient diagram

\xymatrix{
 \mathcal{N}  \ar[d]^{q}  \ar[r]^{i_{\mathcal{N}}} &     \mathcal{S}_{\op{e}}(\Sigma, \, \sigma)\\
\mathcal{O}} 
By Lemma \ref{moser} and \S \ref{smooth},  the quotient induces a smooth structure on $\mathcal{O}=q(\mathcal{N})$. The pullback under the inclusion $i_{\mathcal{N}}$ of the form ${{\omega^{D}}}_{\mathcal{S}_{\op{e}}(\Sigma, \, \sigma)}$ to $\mathcal{N}$ is closed, horizontal (by Lemma \ref{vanish} and Lemma \ref{identifyx}) and invariant to the action of $\Diff(\Sigma,\,\sigma)$, hence basic. Therefore it is the pullback under the quotient $q$ of a closed $2$-form $\omega^{D}_{\red}$ on $\mathcal{O}$. The reduced form is given by
$$ \omega^D_{\red}([v_1],[v_2])=\int_{\Sigma} \omega(v_1,v_2)\sigma.$$
By part (1) of Lemma \ref{tangent}, 
the form $\omega^D_{\red}$ is weakly non-degenerate. 
\end{proof}

\begin{remark} \label{remold}
Similarly we get a closed $2$-form $\omega^{\red}_{\mathcal{S}_{\op{e}}(\Sigma, \, \sigma)}$ on $\mathcal{O}$:
$$ \omega_{\mathcal{S}_{\op{e}}(\Sigma, \, \sigma)}^{\red}([v_1],[v_2])=  \omega_{\mathcal{S}_{\op{e}}(\Sigma, \, \sigma)}(v_1,v_2) ,$$
well defined and weakly non-degenerate by Lemma \ref{tangent}.
By Lemma \ref{exact}, we get $\omega^{\red}_{\mathcal{S}_{\op{e}}(\Sigma, \, \sigma)}=2\omega^D_{\red}$ on $\mathcal{O}$.
However, in that case the $\Ham(M, \, \omega)$-orbits are not integral manifolds of the distribution $\mathcal{D}$ given by 
$$D_i=\{v \in \op{T}_{i} {\mathcal{S}_{\op{e}}(\Sigma, \, \sigma)} \, | \, {\omega_{\mathcal{S}_{\op{e}}(\Sigma, \, \sigma)}}_{i}(v, \xi_{ \mathcal{S}_{\op{e}}(\Sigma, \, \sigma)}(i))=0 \, \forall \xi \in \X(\Sigma, \, \sigma) \}.$$
The reason for the difference is that for $v \in \op{T}_{i} {\mathcal{S}_{\op{e}}(\Sigma, \, \sigma)}$ the form $\iota_{v}{ {{\omega^{D}}}_{\mathcal{S}_{\op{e}}(\Sigma, \, \sigma)}}|_{\X(\Sigma, \, \sigma)}\equiv 0$ iff $v=i^{*}{V_{H}}$ for a Hamiltonian vector field $V_{H}$ on $M$ (by Lemma \ref{vanish} and Lemma \ref{identifyx}) whereas $\iota_{v} { {\omega_{\mathcal{S}_{\op{e}}(\Sigma, \, \sigma)}}}|_{{\X(\Sigma, \, \sigma)}}\equiv 0$ for every $v$ in  $\op{T}\!\mathcal{S}_{\op{e}}(\Sigma, \, \sigma)$, by part (2)  of Lemma \ref{tangent}. From the same reason, the form  $\omega_{\mathcal{S}_{\op{e}}(\Sigma, \, \sigma)}$ is degenerate.
\end{remark}

\section{Corollary to moduli spaces of $J$-holomorphic $\Sigma$-curves}\label{sec4}

For $J \in \J(M,\, \omega)$, denote by
  $$\Ham^{J}(M, \, \omega)$$ the subset of $\Ham(M, \, \omega)$ of $J$-holomorphic Hamiltonian symplectomorphisms.

  \begin{lemma} \label{osubset}
   If $f$ is an embedded $(j,J)$-holomorphic curve $f \colon \Sigma \to M$, and $g = \phi \circ f$ for $\phi \in \Ham^{J}(M, \, \omega)$ then $g \in \M_{\scriptop{e}}(A,\,\Sigma,\,J)$.
  If, in addition, 
   $f^{*}\omega=\sigma$, then   $g^{*}\omega=\sigma$.
  \end{lemma}
  
  \begin{proof}
  If $g = \phi \circ f$ for $\phi \in \Ham^{J}(M, \, \omega)$ then $g$ is an embedding $\Sigma \to M$, and $$\op{d}\!g \circ j= \op{d} \!\phi \circ \op{d} \!f \circ j=\op{d}\!\phi \circ  J \circ \op{d}\!f=J \circ \op{d} \!\phi \circ \op{d}\!f=J \circ \op{d}\!g,$$
  i.e., $g$ is a $(j,J)$-holomorphic $\Sigma$-curve in $M$. If $f^{*}\omega=\sigma$, then
  $$g^{*}\omega=(\phi \circ f)^{*}{\omega}=f^{*}\phi^{*}\omega=f^{*}\omega=\sigma.$$
  \end{proof}

     \begin{lemma} \label{clmod1}
Consider $f_1 \colon \Sigma \to M$ and $f_2 \colon \Sigma \to M$. Assume that  ${f_i}^{*}\omega=\sigma$ for $i=1,2$.
If $\phi$ is a diffeomorphism $\Sigma \to \Sigma$ such that  $f_1 \circ \phi =f_2$,
then $\phi \in \Diff(\Sigma, \, \sigma)$.
\end{lemma}

\begin{proof}
\begin{eqnarray*}
\sigma(\op{d}\!\phi(\cdot),\op{d}\!\phi(\cdot))&=&{f_1}^{*}\omega(\op{d}\!\phi(\cdot),\op{d}\!\phi(\cdot))=\omega(\op{d}\!{f_1}\op{d}\!\phi(\cdot),\op{d}\!{f_1}\op{d}\!\phi(\cdot))\\
                                                                         &=&\omega(\op{d}\!{f_2}(\cdot),\op{d}\!{f_2}(\cdot))={f_2}^{*}\omega(\cdot,\cdot)\\
                                                                         &=&\sigma(\cdot,\cdot)
\end{eqnarray*}
\end{proof}

 \begin{lemma} \label{clmod2}
Consider two $(j,J)$-holomorphic immersions $f_1 \colon \Sigma \to M$ and $f_2 \colon \Sigma \to M$. 
If $\phi$ is a diffeomorphism $\Sigma \to \Sigma$ such that  $f_1 \circ \phi =f_2$,
then $\phi \in \Aut(\Sigma,\,j)$.
\end{lemma}
\begin{proof}
\begin{eqnarray*}
\op{d}\!{f_1} \circ \op{d}\!{\phi} \circ j=\op{d}\!{f_2} \circ j=J \circ \op{d}\!{f_2}=J \circ \op{d}\!{f_1} \circ \op{d}\!{\phi}=\op{d}\!{f_1} \circ j \circ \op{d}\!\phi.
\end{eqnarray*}
Since $f_1$ is an immersion, we conclude $\op{d}\! \phi \circ j=j \circ \op{d}\!\phi$.
\end{proof}

\begin{lemma} \label{integcomp}
 Let $J$ be an almost complex structure on $M$, and $f \colon \Sigma \to M$ an immersed $(j,J)$-holomorphic curve. 
Assume that $J$ is $\omega$-compatible. If  $v$ is \emph{not} everywhere tangent to $f(\Sigma)$, then $\omega^D_{\red}([v],[\tilde{J}v])\neq 0$.
\end{lemma}

\begin{proof}
Decompose $v= \xi_v+\tau_v$, with
 $\xi_v$ everywhere $\omega$-orthogonal to $\Sigma$ and $\tau_v$ everywhere tangent to $\Sigma$. 
 By assumption $\xi_v \neq 0$. Then $\tilde{J}v=\tilde{J}\xi_v+\tilde{J}\tau_v$ with $\tilde{J}\xi_v \neq 0$. By Lemma \ref{jclaim}, $[\tilde{J}v]=[\tilde{J}\xi_v] \neq 0$. 
 Thus, since $J$ is $\omega$-tamed and $\sigma$ is an area form, $\omega^D_{\red}([v],[\tilde{J}v])=\omega^D_{\red}([\xi_v],[\tilde{J}\xi_v])=\int_{\Sigma} \omega(\xi_v,J(\xi_v)) \, \sigma \neq 0$.
 
\end{proof}

\begin{proof}[Proof of Corollary \ref{mod}]
    $\Ham^{J}(M, \, \omega)$ is a closed subgroup of $\Ham(M, \, \omega)$.
    (Notice that a diffeomorphism $\phi \colon M \to M$ is $J$-holomorphic iff $\phi^{-1}$ is.)  
       Since  $J$ is regular for the projection map $p_{A} \colon \M_{\scriptop{i}}(A,\,\Sigma, \,\mathcal{J}) \to \J,$ the space $\M_{\scriptop{e}}(A,\,\Sigma,\,J)$  of embedded $(j,J)$-holomorphic $\Sigma$-curves in a homology class $A \in \op{H}_2(M,\,\Z)$ is a finite-dimensional manifold  \cite[Thm 3.1.5]{MS2}.
  By Lemma \ref{osubset}, the group  $\Ham^{J}(M, \, \omega)$ acts on $\M_{\scriptop{e}}(A,\,\Sigma,\,J)$ by composition on the left.    Since the action of $\Ham^{j}(M,\, \omega)$ on  $\op{C}^{\infty}(\Sigma,\, M)$ preserves both $\omega^D_{(\Sigma, \, \sigma)}$ and the almost complex structure $ \tilde{J} \colon \op{T}\!  \op{C}^{\infty}(\Sigma,\, M) \rightarrow  \op{T}\! \op{C}^{\infty}(\Sigma,\, M)$ induced from 
     $J \colon \op{T}M \to \op{T}M$, it also preserves the map defined by 
     $$\tilde{g}(\tau_1,\tau_2)=\omega^{D}_{(\Sigma, \, \sigma)}(\tau_1, \tilde{J}\tau_2).$$
     By Lemma \ref{jinteg}, the manifold $\M_{\op{e}}(A, \,\Sigma,\,J)$ is closed under $\tilde{J}$, hence by Lemma \ref{integcomp2}, the restriction of $\omega^{D}_{(\Sigma, \, \sigma)}$ to  $\M_{\op{e}}(A, \,\Sigma,\,J)$ is non-degenerate, so the map $\tilde{g}$ restricts to a metric on $\M_{\op{e}}(A, \,\Sigma,\,J)$, and   $\Ham^{J}(M, \, \omega)$ acts on the moduli space as a subgroup of the isometry group. Since the action of the isometry group on a finite-dimensional manifold is proper and $\Ham^{J}(M, \, \omega)$ is a closed subgroup, every orbit of $\Ham^{J}(M, \, \omega)$ is an embedded submanifold of $\M_{\op{e}}(A, \,\Sigma,\,J)$.
     
     Let $\mathcal{N}$ be an orbit of $\Ham^{J}(M, \, \omega)$ through an embedded $(j,J)$-holomorphic curve $f \colon \Sigma \to M$ for which $f^{*}\omega=\sigma$. By Lemma \ref{osubset}, the orbit $\mathcal{N} \subset \M_{\op{e}}(A, \,\Sigma,\,J) \cap \mathcal{S}_{\op{e}}(\Sigma, \, \sigma)$. Thus, by Lemma \ref{vanish} and Lemma \ref{tangent}, the reductions  $\omega^D_{\red}$ and $\omega^{\red}_{\mathcal{S}_{\op{i}}(\Sigma)}$ are well defined on the quotient  $\mathcal{N}$ modulo $\Diff(\Sigma, \, \sigma)$.
  By Lemmata \ref{clmod1} and \ref{clmod2}, $\mathcal{N}$ modulo $\Aut(\Sigma, \, j)$ is the same as $\mathcal{N}$ modulo $\Diff(\Sigma, \, \sigma)$.
   By Lemma \ref{integcomp},  the form $\omega^D_{\red}$ restricted to $\mathcal{N}$
 modulo $\Aut(\Sigma, \, j)$ is non-degenerate. 
 By Lemma \ref{exact}, 
  the form $\omega^{\red}_{\mathcal{S}_{\op{i}}(\Sigma)}$ coincides with the form $2\omega^D_{\red}$ on the quotient.
\end{proof}

\begin{remark} \label{remregst}
For examples of regular integrable compatible almost complex structures, we look at K\"ahler manifolds whose automorphism groups act transitively. By \cite[Proposition 7.4.3]{MS2}, if $(M, \, \omega_0,\,J_0)$ is a compact K\"ahler manifold and $G$ is a Lie group that acts transitively on $M$ by holomorphic diffeomorphisms, then $J_0$ is regular for every $A \in \op{H}_2(M,\,\Z)$. This applies, e.g., when $M = \C P^n$, $\omega_0$ the Fubini-Study form, $J_0$ the standard complex structure on $\C P^n$, and $G$ is the automorphism group $\PSL(n+1)$, 
\end{remark}

{\bf Acknowledgements}. 
I thank Yael Karshon for many helpful discussions. I am also grateful to Dusa McDuff for valuable observations.

\appendix

\section{Appendix : a symplectic form and compatible almost complex structures on the space of immersed surfaces} \label{sec1}

To prove the closedness part of Proposition \ref{thclonon2}, we first show the following claim.

\begin{claim} \label{clpi}

For $\tau_1, \, \tau_2 \colon \Sigma \rightarrow f^{\ast}(\op{T}\! M)$, we have 
$$\omega(\tau_1, \, \tau_2) \, \sigma = \frac{1}{2} \iota_{(\tau_1,0), \, (\tau_2,0)} (\op{ev}^{\ast} ( \omega))  \wedge  (\pi_{\Sigma}^{\ast}(\sigma)),$$
where $\pi_{\Sigma} \colon \op{C}^{\infty}(\Sigma, \, M) \times \Sigma \to \Sigma$ is the projection 
onto the second coordinate. 

\end{claim}

\begin{proof}

Indeed, 
 $$
 \iota_{(\tau_1,0), \, (\tau_2,0)}(\op{ev}^{\ast} ( \omega))  \wedge  (\pi_{\Sigma}^{\ast}(\sigma))(\cdot,\cdot)=
  2 \omega(\tau_1,\tau_2) \sigma(\cdot,\cdot) + 2 \omega(\tau_1, \op{d}\!f(\cdot)) \wedge \sigma (\op{d}\!\pi_{\Sigma}(\tau_2,0),\cdot).
 $$
 Notice that the second summand in the right-hand term always vanishes. 
 
 \end{proof}

\begin{proof}[Proof of the closedness part of Proposition \ref{thclonon2}]

We will show that for any
  two surfaces $R_1$ and $R_2$ in $\mathcal{S}_{\op{i}}(\Sigma)$, that are homologous relative
  to a common boundary $\partial R$,
  \begin{eqnarray} \label{xxx}
  \int_{R_1} {\omega^{D}}_{(\Sigma, \, \sigma)} = \int_{R_2} 
  {\omega^{D}}_{(\Sigma, \, \sigma)}.
  \end{eqnarray}
 See also  \cite[Prop.\ 12]{lee} for a proof that $\op{d}\! {\omega^{D}}_{(\Sigma, \, \sigma)} =0 $ in the Convenient Setup.
 
 Let $R_1$ and $R_2$ be two surfaces that are homologous relative
  to a common boundary $\partial R$.
 By the above claim,
  $$
  \int_{R_i} {\omega^{D}}_{(\Sigma, \, \sigma)} = \frac{1}{2} \int_{R_i \times \{x\}} \left( \int_{\{f\} \times
     \Sigma} (\op{ev}^{\ast} ( \omega))  \wedge  (\pi_{\Sigma}^{\ast}(\sigma))  \right) =
     \frac{1}{2} \int_{R_i \times \Sigma} (\op{ev}^{\ast} ( \omega))  \wedge  (\pi_{\Sigma}^{\ast}(\sigma)).
 $$
Since $R_1 \times \Sigma$ is homologous to $R_2 \times \Sigma$ relative to the
  boundary $\partial R \times \Sigma$, and $(\op{ev}^{\ast} ( \omega))  \wedge  (\pi_{\Sigma}^{\ast}(\sigma))$ is closed, we have that:
  $$
  \int_{R_1 \times \Sigma} (\op{ev}^{\ast} ( \omega))  \wedge  (\pi_{\Sigma}^{\ast}(\sigma))
      = \int_{R_2 \times \Sigma}(\op{ev}^{\ast} ( \omega))  \wedge  (\pi_{\Sigma}^{\ast}(\sigma)).
     $$
 Therefore,  we get (\ref{xxx}).
 
 \end{proof}

 \subsection*{Compatible almost complex structures}

An \emph{almost complex structure} on a manifold $M$ is an automorphism of
the tangent bundle, $$J \colon \op{T}\!M \to \op{T}\!M,$$ such that $J^2 = -\op{Id}$.
The pair $(M,\,J)$
is called an \emph{almost complex manifold}. 

An almost complex structure is
\emph{integrable} if it is induced from a complex manifold structure. In
dimension two any almost complex manifold is integrable (see, e.g.,
\cite[Th.~4.16]{MS}). In higher dimensions this is not true
\cite{calabi}.

\begin{definition} \label{defj}
  Let $J$ be an almost complex structure on $M$. We define a map
  $$
  \tilde{J}: \op{T}\!  \op{C}^{\infty}(\Sigma,\, M) \rightarrow 
  \op{T}\! \op{C}^{\infty}(\Sigma,\, M)
  $$
   as follows:
for $\tau \colon \Sigma \rightarrow f^{\ast} (\op{T}\!M)$, the vector $\tilde{J}(\tau)$
  is the section
  $
  \hat{J} \circ \tau,
  $
 where $\hat{J}$ is the map defined by the commutative diagram
   \begin{eqnarray} 
\xymatrix{ \ar @{} [dr] |{\circlearrowleft}
f^*(\op{T}\!M)  \ar[r]^{\hat{J}}      \ar[d]  &  f^*(\op{T}\!M)
                  \ar[d]   \\
                  \op{T}\!M \ar[r]^J   &       \op{T}\!M } \nonumber
\end{eqnarray}     
 \end{definition}
Due to the properties of the almost complex structure $J$, the map $\tilde{J}$ is an automorphism and ${\tilde{J}}^{2}=  -\op{Id}$.
\begin{claim} 
 Let $J$ be an almost complex structure on $M$.
 Then $\tilde{J}$ is an almost complex structure on $\op{C}^{\infty}(\Sigma,\, M)$. 
\end{claim}

 An almost complex
structure $J$ on $M$ is 
\emph{tamed} by a symplectic form $\omega$ if  $\omega_{x}(v,\,Jv) >0$ for all non-zero $v \in \op{T}_x M$. If, in addition, $\omega_{x}(Jv,\,Jw)=\omega_{x}(v,\,w)$ for all $v, \, w \in T_x M$, we say that $J$ is \emph{$\omega$-compatible}. 
The space $\J=\mathcal{J}(M, \, \omega)$ of $\omega$-compatible almost complex structures is non-empty and contractible, in particular path-connected \cite[Prop.~4.1]{MS}.

\begin{lemma} \label{integcomp2}
Let $J$ be an almost complex structure on $M$.
\begin{enumerate}
\item If $J$ is $\omega$-tamed, then the induced almost complex structure $\tilde{J}$ on the space of immersions $\Sigma \to M$  is $ {\omega^{D}}_{(\Sigma, \, \sigma)} $-tamed. 
\item If $J$ is $\omega$-compatible, then $\tilde{J}$ is $  {\omega^{D}}_{(\Sigma, \, \sigma)}$-compatible. 
\end{enumerate}
\end{lemma}

\begin{proof}
\begin{enumerate}
\item
For a non-zero vector field $\tau \colon \Sigma \to f^{*}\op{T}M$,
   \begin{eqnarray} 
        {{\omega^{D}}_{(\Sigma, \, \sigma)}}_f( \tau, \,\tilde{J}(\tau)):=\int_{\Sigma} \omega(\tau,J(\tau)) \, \sigma > 0, \nonumber
  \end{eqnarray}
 the last inequality is since $J$ is $\omega$-tamed, $f$ is an immersion, and $\sigma$ is an area-form.

\item
For $\tau_1, \, \tau_2 \colon \Sigma \to  f^{*}\op{T}\!M$,
 \begin{eqnarray} 
        {{\omega^{D}}_{(\Sigma, \, \sigma)}}_f( \tilde{J}(\tau_1), \,\tilde{J}(\tau_2)):=\int_{\Sigma} \omega(J (\tau_1), \, J(\tau_2)) \, \sigma = \int_{\Sigma} \omega(\tau_1, \,\tau_2) \, \sigma=  { {\omega^{D}}_{(\Sigma, \, \sigma)}}_f(\tau_1, \,\tau_2) , \nonumber
  \end{eqnarray}
since $J$ is $\omega$-compatible. 

\end{enumerate}
\end{proof}

\begin{proof}[Proof of the non-degeneracy part of Proposition \ref{thclonon2}]
It follows from part (1) of Lemma \ref{integcomp2}, and the fact that the space of $\omega$-tamed structures is not empty. 
\end{proof}

\bigskip

 Fix $\Sigma=(\Sigma, \, j)$. 
A smooth ($\op{C}^{\infty}$) curve $f \colon \Sigma \rightarrow M$ 
is called
$J$\--\emph{holomorphic} if the differential $\op{d}\!f$ is a complex linear map
between the fibers $\op{T}_{p}(\Sigma)\rightarrow \op{T}_{f(p)}(M)$ for all
$p \in
\Sigma$, i.e. 
\[\op{d}\!f_{p}\circ {j}_{p} = {J}_{f(p)}\circ
\op{d}\!f_{p}.\]
A $J$-holomorphic curve is 
 \emph{simple} if it cannot be factored through a branched covering of $\Sigma$.
  The moduli space 
$\M_{\scriptop{i}}(A,\,\Sigma,\,J)$ of simple immersed $J$-holomorphic $\Sigma$-curves in a homology class $A \in \op{H}_2(M,\,\Z)$. 
 We look at almost complex structures that are regular for the projection map $$p_{A} \colon \M_{\scriptop{i}}(A,\,\Sigma, \,\mathcal{J}) \to \J;$$ for such a  $J$, the space $\M_{\scriptop{i}}(A,\,\Sigma,\,J)$ is a finite-dimensional manifold \cite[Thm. 3.1.5]{MS2}.

 \begin{lemma} \label{jinteg}
If $J$ is an integrable almost complex structure on $M$ that is regular for $A$, and $f \colon \Sigma \to M$ is a $J$-holomorphic map in $A$, then for $\tau \in \op{T}_{f}\M(A, \,\Sigma,\,J)$, the vector $J \circ \tau$ is also in  $\op{T}_{f}\M(A, \,\Sigma,\,J)$.
\end{lemma}
 
 For proof, see, e.g., \cite[Lem. 3.3]{ckp}.

As a result of Lemma \ref{jinteg} and Lemma \ref{integcomp2}, we get the following proposition.
\begin{proposition} \label{intregsym}
  If $J$ is an integrable almost complex structure on $M$ that is compatible
  with $\omega$ and regular for $A$, then $ {\omega^{D}}_{(\Sigma, \, \sigma)}$ restricted to
  $\M_{\op{i}}(A,\,\Sigma,\,J) $ is symplectic.
\end{proposition}

\begin{lemma} \label{jclaim}
 Let $J$ be an almost complex structure on $M$.
Assume that $f \colon \Sigma \to M$ is $J$-holomorphic. Then, at $x \in \Sigma$, 
\begin{enumerate}
   \item[\textup{1.}] if $\tau_x \in \op{d}\!f_{x}( \op{T}_{x}\Sigma )$ then $J_{f(x)} (\tau_x) \in \op{d}\!f_{x}( \op{T}_{x}\Sigma )$; 
   \item[\textup{2.}] if $J$ is $\omega$-compatible and $\phi_x$ is $\omega$-orthogonal to  $\op{d}\!f_{x}( \op{T}_{x}\Sigma )$, then so is $J_{f(x)} (\phi_x)$.
 \end{enumerate}  
\end{lemma} 

\begin{proof}
\begin{enumerate}
\item By assumption $\tau_x=\op{d}\! f_{x}(\alpha)$ for $\alpha \in \op{T}_x \Sigma$. Hence, since $f$ is $J$-holomorphic,
 $$J_{f(x)} (\tau_x)=J_{f(x)} (\op{d}\!f_{x}(\alpha))=\op{d}\! f_{x}(j_{x} \alpha).$$ 
\item By the previous item, $J_{f(x)} (\op{d}\!f_{x}( \op{T}_{x}\Sigma )) \subseteq \op{d}\!f_{x}( \op{T}_{x}\Sigma )$, hence, since $J^2=-\id$, $$J_{f(x)}(\op{d}\!f_{x}( \op{T}_{x}\Sigma ))=\op{d}\!f_{x}( \op{T}_{x}\Sigma ).$$
Let $\tau_x \in \op{d}\!f_{x}( \op{T}_{x}\Sigma )$, then there exists $\tau'_x \in \op{d}\!f_{x}( \op{T}_{x}\Sigma )$ such that $\tau_x=J_{f(x)}(\tau'_x)$. By assumption, $\omega(\phi_x, \tau'_x)=0$. Thus $$\omega(J_{f(x)}(\phi_x),\tau_x)=
\omega(J_{f(x)}(\phi_x),J_{f(x)}(\tau'_x))=\omega(\phi_x,\tau'_x)=0.$$

\end{enumerate}
\end{proof}

\begin{corollary} \label{corj}
Let $J$ be an $\omega$-tamed almost complex structure on $M$. Assume that $f \colon \Sigma \to M$ is $J$-holomorphic.  
Then  every $\mu \in \op{T}_f(\op{C}^{\infty}(\Sigma, \, M))$ can be uniquely decomposed as $$\mu=\mu'+\mu'',$$ where 
 $\mu' (x) \in  \op{d}\!f_{x} (\op{T}_{x} \Sigma)$ at every $x \in \Sigma$, and $\mu''(x)$ is $\omega$-orthogonal to  $\op{d}\!f_{x}( \op{T}_{x}\Sigma )$ at every $x \in \Sigma$. 
\end{corollary}

\begin{proof}
At $x \in \Sigma$, if $ v \in W_x=\op{d}\!f_{x}( \op{T}_{x}\Sigma )$, then $J(v) \in W_x$ (by part (1) of Lemma \ref{jclaim}). Since $J$ is $\omega$-tamed, if $v \neq 0$,
$\omega(v,J(v))>0$, hence $0 \neq v \in W_x$ is not in  ${W_x}^{\omega}$. Thus $W_x \cap {W_x}^{\omega}=\{0\}$.
Since $\dim {W_x}^{\omega}=\dim M - \dim W_x$,  we deduce that  $\op{T}_{f(x)}M=W_x \oplus {W_x}^{\omega}$.

To conclude the corollary, recall that if a bundle $E \to B$ equals the direct sum of sub-bundles $E_1 \to B$ and $E_2 \to B$, then the space of sections of $E$ equals the direct sum of the space of sections of $E_1$ and the space of sections of $E_2$.

\end{proof}

\begin{remark} \label{extend}
Fix a symplectic form $\sigma$ on and a complex structure $j$ that is $\sigma$-compatible on $\Sigma$. Given a symplectic embedding $f \colon (\sigma,\Sigma) \to (M, \, \omega)$ there is a $J \in \J(M, \, \omega)$ such that $f$ is $(j,J)$-holomorphic. 
Define $J|_{\op{T}\!(f(\Sigma))}$ such that $f$ is holomorphic.
Extend it to a compatible fiberwise complex structure on the 
symplectic vector bundle $\op{T}\!{M}|_{f(\Sigma)}$.  
Then extend it to a compatible almost complex structure on $M$. 
See \cite[Section 2.6]{MS}. 
\end{remark}

\noindent

\bigskip\noindent
Liat Kessler\\
Mathematics Department, Technion\\
Technion city, Haifa 32000, Israel\\
{\em E\--mail}: lkessler@tx.technion.ac.il\\

\end{document}